\documentclass[reqno]{amsart}
\usepackage{amssymb}
\usepackage[usenames, dvipsnames]{color}
\usepackage{todonotes}

\usepackage{txfonts}
\usepackage{hyperref}
\usepackage{verbatim}
\usepackage[pagewise]{lineno}

\usepackage[normalem]{ulem}
\usepackage{cancel}
\usepackage{soul}

\theoremstyle{plain}
\newtheorem{theorem}{Theorem}[section]
\newtheorem{lemma}[theorem]{Lemma}
\newtheorem{corollary}[theorem]{Corollary}

\theoremstyle{definition}
\newtheorem{definition}[theorem]{Definition}

\theoremstyle{remark}
\newtheorem{remark}[theorem]{Remark}

\newtheorem{claim}[theorem]{Claim}

\def\sgn{\text{sgn }}

\def\p{\partial}

\def\be{\begin{equation}}
\def\ee{\end{equation}}
\def\bes{\begin{equation*}}
\def\ees{\end{equation*}}
\def\bali{\begin{aligned}}
	\def\eali{\end{aligned}}

\renewcommand{\epsilon}{\varepsilon}

\numberwithin{equation}{section}

\makeatletter
\def\dashint{\operatorname%
	{\,\,\text{\bf--}\kern-.98em\DOTSI\intop\ilimits@\!\!}}
\makeatother

\begin{document}

	\title[Time analyticity for nonlocal parabolic equations]{Time analyticity for nonlocal parabolic equations}

\author[H. Hong]{Hongjie Dong}
\address[H. Dong]{Division of Applied Mathematics, Brown University, Providence, RI 02912, USA}

\email{hongjie\_dong@brown.edu}

\author[C. Zeng]{Chulan Zeng}
\address[C. Zeng]{Department of Mathematics, University of California, Riverside, CA, 92521, USA}

\email{czeng011@ucr.edu}

\author[Q. Zhang]{Qi S. Zhang}
\address[Q. Zhang]{Department of Mathematics, University of California, Riverside, CA 92521, USA}

\email{qizhang@math.ucr.edu}
\thanks{}

\subjclass[2010]{}


\begin{abstract}
In this paper, we investigate pointwise time analyticity of solutions to nonlocal parabolic equations in the settings of $\mathbb{R}^d$ and a complete Riemannian manifold $\mathrm{M}$. On one hand, in $\mathbb{R}^d$, we prove that any solution $u=u(t,x)$ to $u_t(t,x)-\mathrm{L}_\alpha^{\kappa} u(t,x)=0$, where $\mathrm{L}_\alpha^{\kappa}$ is a nonlocal operator of order $\alpha$, is time analytic in $(0,1]$ if $u$ satisfies the growth condition $|u(t,x)|\leq C(1+|x|)^{\alpha-\epsilon}$ for any $(t,x)\in (0,1]\times \mathbb{R}^d$ and $\epsilon\in(0,\alpha)$. We also obtain pointwise estimates for $\p_t^kp_\alpha(t,x;y)$, where $p_\alpha(t,x;y)$ is the fractional heat kernel. Furthermore, under the same growth condition, we show that the mild solution is the unique solution.
On the other hand, in a manifold $\mathrm{M}$, we also prove the time analyticity of the mild solution under the same growth condition and the time analyticity of the fractional heat kernel, when $\mathrm{M}$ satisfies the Poincar\'e inequality and the volume doubling condition. Moreover, we also study the time and space derivatives of the fractional heat kernel in $\mathbb{R}^d$ using the method of Fourier transform and contour integrals. We find that when $\alpha\in (0,1]$, the fractional heat kernel is time analytic at $t=0$ when $x\neq 0$, which  differs from the standard heat kernel.

As corollaries, we obtain  sharp  solvability  condition for the backward nonlocal parabolic equations and time analyticity of some nonlinear nonlocal parabolic equations with power nonlinearity of order $p$.
These results are related to those in \cite{[DZ]} and \cite{[Zeng]} which deal with local equations.

\end{abstract}

\maketitle

\tableofcontents

	\section{Introduction}
	In this paper, we investigate pointwise time analyticity of solutions to nonlocal parabolic equations in the settings of $\mathbb{R}^d$ and a complete Riemannian manifold $\mathrm{M}$ satisfying the standard  Conditions \eqref{cond1} and \eqref{cond2}.  One of our main results reads that the fractional heat kernel on $\mathbb{R}^d$ is time analytic at $t=0$ when $x\neq 0$ and $\alpha\in (0,1]$, which differs from the standard heat kernel.  As an intermediate result, we obtain the uniqueness of solutions to nonlocal parabolic equations in $\mathbb{R}^d$, which improves a result in \cite{[Chen]} in the sense that instead of the bound $C t/\left(t^{{1}/{\alpha}}+|x|\right)^{d+\alpha}$, we only impose the growth condition $|u(t,x)|\leq C(1+|x|)^{\alpha-\epsilon}$ for any $(t,x)\in (0,1]\times \mathbb{R}^d$ and $\epsilon\in(0,\alpha)$.
	In the manifold setting , we obtain  lower and upper bounds for the fractional heat kernel $p_\alpha$, and  prove that $p_\alpha$ is time analytic for any $(t,x)\in (0,\infty)\times \mathrm{M}$.  These results allow us to solve the solvability problem of the backward nonlocal parabolic equations which can be ill-posed.

Before presenting the results in detail, we wish to justify their value by recalling a number of related results in the literature and describing some new applications.
The study of analyticity property of solutions to PDEs 	has been a classical topic. Even though the spatial analyticity is usually
true for generic solutions of the heat equation, the time analyticity is harder to prove and is false in general. For instance, it is not hard to construct a solution of
the heat equation in a space-time cylinder in the Euclidean setting, which is not time analytic in a sequence of moments.
In fact, the time analyticity is not a local property, rather it requires certain boundary or growth conditions on the solutions. There is a vast literature on
 time-analyticity for the heat equation and other parabolic type
equations under various assumptions. See, for example, \cite{[Masuda]}, \cite{[Komatsu]}, \cite{[GI]}, \cite{[LE]}, \cite{[Widder]}, \cite{[DZ]}, \cite{[Zhang2]}, and \cite{[Zeng]} and the citations therein.  One can also consider
solutions in certain $L^p$ spaces with $p\in(1,\infty)$. See \cite{[Pr]} for a large class of dissipative equations in the periodic setting.  We also mention that in \cite{[LE]}, for any bounded domain $\Omega\subset \mathbb{R}^d$ with analytic boundary, the authors proved that any solution of the high order heat equation
	$$
	\left\{\begin{array}{l}
	u_t+(-\Delta)^m u=0, \quad\forall (t,x)\in (0,1]\times \Omega, \\
	u=Du=\cdots=D^{m-1}u=0\ \text{on}\  (0,1]\times \partial\Omega,\ u(0,x)\in L^2(\Omega)
	\end{array}\right.
	$$
	is time analytic in $t\in(0,1]$.

Recently new applications of time analyticity are found in control theory and in the study of backward equations which is essential in stochastic analysis and mathematical finance. A fundamental fact in control theory for heat type equations is that if a state is reachable by the free equation then it is  reachable by suitable control from any reasonable initial value. The former is equivalent to the solvability of the free backward equation from this state. However
	this backward solvability question has been vexing the control theory community for years.  As a matter of fact,
	in a recent paper \cite{[La]}, it was written:"However, it is a quite hard task to decide whether a given state is the value at some
time of a trajectory of the system without control (free evolution). In practice, the only
known examples of such states are the steady states." This problem for the heat equation was solved in \cite{[DZ]} not long ago. More precisely, in the paper \cite{[DZ]} (see also \cite{[Zhang2]}), it was proved that if a smooth solution of the heat equation in $(-2,0]\times\mathrm{M}$ is of exponential growth of order 2, then it is time analytic  in $t\in[-1,0]$.  Here $\mathrm{M}$ is either the Euclidean space or  certain noncompact manifolds.  Also, an explicit condition is found on the solvability of the backward heat equation from a given time, which is equivalent to
the time analyticity of the solution of the heat equation at that time.
 Lately, the time analyticity of solutions to the biharmonic heat equation, the heat equation with potentials, and some nonlinear heat equations are proven in \cite{[Zeng]}.
See also \cite{[DP]} for other results about time analyticity of parabolic type differential equations in the half space.
 One of the goals of this paper is to extend the result to nonlocal parabolic equations which have attracted intensive research (See Corollary \ref{cor1}).

Now let us present the main results formally. For clarity, we will first treat the nonlocal parabolic equations in  the setting of $\mathbb{R}^d$, which reads

	\be\label{PDE}
	u_t(t,x)-\mathrm{L}_\alpha^{\kappa} u(t,x)=0,\ \alpha\in (0,2),\ (t,x)\in [0,1]\times \mathbb{R}^d,
	\ee
	where $\mathrm{L}_\alpha^{\kappa}$ is a nonlocal elliptic operator defined as follows.

	\begin{definition}
	We define
	\be\label{def}	\mathrm{L}_\alpha^{\kappa} f(x):=p.v.\int_{\mathbb{R}^d}(f(x+z)-f(x))\frac{\kappa(x,z)}{|z|^{d+\alpha}}\,dz,\ee
	where $p.v.$ means the principal value.
	Here $\kappa=\kappa(x,z)$ on $\mathbb{R}^d\times \mathbb{R}^d$ is a measurable function	satisfying that
	\be\label{1stco} 0<\kappa_0\leq\kappa(x,z)\leq \kappa_1,\ \kappa(x,z)=\kappa(x,-z),\ee
	and for a constant $\beta\in(0,1)$,
	\be\label{2ndco} |\kappa(x,z)-\kappa(y,z)|\leq \kappa_2|x-y|^\beta,\ee
where $\kappa_0$, $\kappa_1$, and $\kappa_2$ are positive constants.
\end{definition}

The fraction Laplacian $(-\Delta)^{\alpha/2}$ is a typical example of $\mathrm{L}_\alpha^{\kappa}$.
As a special case, we also obtain the time and space derivative estimates of the fractional heat kernel $p_\alpha(t,x)$ of
\be\label{PDE3}
u_t(t,x)+\left(-\Delta\right)^{{\alpha}/{2}} u(t,x)=0,\ \alpha\in (0,2),\ (t,x)\in [0,1]\times \mathbb{R}^d.
\ee

Our results involve both solutions and fractional heat kernels. We say that a function  $p_\alpha(t,x;y)$ is a fractional heat kernel of the equation (\ref{PDE}) in $\mathbb{R}^d$, if
$$\p_tp_\alpha(t,x;y)=\mathrm{L}_\alpha^{\kappa}p_\alpha(t,x;y),\quad \lim\limits_{t\searrow 0}p_\alpha(t,x;y)=\delta(x,y).$$
In \cite{[Chen]}, it was proved that the fractional heat kernel is unique under the condition that
$$|p_\alpha(t,x;y)|\leq \frac{C t}{\left(t^{{1}/{\alpha}}+|x-y|\right)^{d+\alpha}},$$
for a constant $C$. In Lemma \ref{le2}, we improve this uniqueness result by only requiring the growth condition \eqref{grow}. The definition of the fractional heat kernel $p_\alpha(t,x;y)$ on a manifold $\mathrm{M}$ will be given in Section \ref{4thsec}.

The next four theorems are the main results of this paper. The first one is a time analyticity result in the case of $\mathbb{R}^d$.
\begin{theorem}\label{theo01}
(a)	Let $p_\alpha(t,x;y)$ be the  heat kernel of equation (\ref{PDE}). Then there exists a positive constant  $C$ such that for any $t\in(0,1]$ and any nonnegative integer $k$,
	\be\label{goal}|\p_t^k p_\alpha(t,x;y)| \leq \frac{ C^{k+1} k^k}{t^{k-1}}\frac{1}{\left(t^{1/\alpha}+|x-y|\right)^{d+\alpha}}.
	\ee
	
(b) Assume that $u=u(t,x)$ is a solution to \eqref{PDE} with polynomial growth of order $\alpha-\epsilon$, i.e.,
	\be\label{grow}
	|u(t,x)|\leq C_1 \left(1+|x|^{\alpha-\epsilon}\right),\ \ \forall (t,x)\in [0,1]\times\mathbb{R}^d,\ 0<\alpha<2,\ \epsilon\in (0,\alpha)
	\ee
	for a positive constant $C_1$.
	Then
	$$u(t,x)=\int_{\mathbb{R}^d}p_\alpha(t,x-y)u(0,y)\, dy$$
	 is the unique smooth solution with initial data $u(0,\cdot)$. Moreover, $u$ is time analytic for any $t\in (0,1]$ with the radius of convergence being independent of $x$.

(c) For any $t\in(1-\delta,1]$ with a small $\delta>0$, we have
	$$
	u(t,x)=\sum_{j=0}^{\infty} a_{j}(x) \frac{(t-1)^{j}}{j !},
	$$
	where $a_0(x)=u(1,x)$, $a_{j+1}(x)=\mathrm{L}_\alpha^{\kappa} a_{j}(x)$,
	$$
	\left|a_{j}(x)\right|=\left|\left({\mathrm{L}_\alpha^{\kappa}} \right)^j a_{0}(x)\right| \leq C_1 C_2^j j^{j} \left(1+|x|^{\alpha-\epsilon}\right), \quad j=0,1,2, \ldots,
	$$
	and $C_2$ is a positive constant.
\end{theorem}
\begin{remark}
	The estimate $|a_{j}(x)|$ in part (c) of this theorem will be used for the solvability of the backward nonlocal parabolic equations and the time analyticity at $t=0$ in the last section.
\end{remark}
\begin{remark}
From the proof of this theorem, for a constant $C>0$, we have
	\be\label{goal1}
	|\p_t^k u(t,x)|\leq \frac{C^{k+1} k^k}{ t^{k-1}} \Big(\frac {1+|x|^{\alpha-\epsilon}} t+\frac 1 {t^{\varepsilon/\alpha}}\Big),\, \forall t\in(0,1]
	\ee
under the growth condition \eqref{grow}.
\end{remark}

Now let us focus on the heat kernel of the fractional Laplacian $(-\Delta)^{\frac{\alpha}{2}}$ in $\mathbb{R}^d$.
Recall that the fractional heat kernel $p_\alpha(t,x)$ for $u_t+(-\Delta)^{\alpha/2}u(t,x)=0$ is given by
\be\label{de}
p_\alpha(t,x)=C(d,\alpha)  \int_{\mathbb{R}^d}e^{-t|\xi|^\alpha}e^{i\xi x}d\xi,
\ee
which can be deduced by the Fourier transform.
\begin{theorem}\label{theo02}
	
	The following statements are true for the fractional heat kernel $p_\alpha(t,x)$ when $t\geq 0$.
	
	(a) For any $\alpha>0$ and for any positive integer $k$, there exist positive constants  $C$, $C_1$, and $C_2$ such that
	\be\label{time}|\p_t^k p_\alpha(t,x)|\leq  \min\left\{\frac{C_1 C_2^{k\alpha} (k\alpha)^{k\alpha}}{|x|^{k\alpha+d}},\frac{C}{t^{k+d/\alpha}}\Gamma\left(\frac{k\alpha+d}{\alpha}\right)\right\},
	\ee
	which implies that $p_\alpha$ is of Gevrey class in time of order $\alpha$ when $x \neq 0$ and $p_\alpha$ is analytic in time  when $t> 0$. Moreover, if $0< \alpha \le 1$ and $x \neq 0$, then $p_\alpha$ is analytic in time for all $t \geq 0$. Here $\Gamma$ is the gamma function.
	
	(b)
	For  any $\alpha>0$ and for any positive integer $k$,
	\be\label{x}
	|\p^k_x p_\alpha(t,x)| \leq \min\left\{\frac{C_1 C_2^{k+\alpha} (k+\alpha)^{k+\alpha} t}{|x|^{\alpha+k+d}},\frac{C}{t^{(k+d)/\alpha}}\Gamma\left(\frac{k+d}{\alpha}\right)\right\},
	\ee
	which implies that $p_\alpha$ is analytic in space at $|x|\neq 0$. 	Especially, when $t\neq0$, $p_\alpha$ is of Gevrey class with order ${1}/{\alpha}$ in space for any $x$.	
\end{theorem}

Part (a) of the theorem shows that for any $\alpha \in (0, 1]$, the fractional heat kernel is time analytic down to $t=0, x \neq 0$, which is not true for the standard heat kernel.

By the above Theorem \ref{theo02}, we have
\begin{corollary}\label{add2}
	If the unique smooth solution $u=u(t,x)$ to the fractional heat equation \eqref{PDE3} satisfies the growth condition \eqref{grow} for some $\alpha\in [1,2)$, then it is analytic in space for any $(t,x)\in(0,1]\times \mathbb{R}^d$.  Moreover, when $\alpha\in (0,1)$, $u$ is of Gevrey class of order $1/\alpha$ in space for any $(t,x)\in(0,1]\times \mathbb{R}^d$.
\end{corollary}

The last two theorems of the paper are in the setting of a complete Riemannian manifold $\mathrm{M}$. We impose the following two  standard conditions on $\mathrm{M}$:

Condition (1): There exists a constant $C_0>0$ such that for any ball $B(x_0,r)$, $x_0\in \mathrm{M}$, $r>0$, and $f\in C^\infty(B(x_0,r))$,
\be\label{cond1}
\int_{B(x_0,r)}|f-f_{B(x_0,r)}|^2\, dx \leq C_0 r^2\int_{B(x_0,r)}|\nabla f|^2\, dx,
\ee
where
$$
f_{B(x_0,r)}:=\frac{1}{|B(x_0,r)|}\int_{B(x_0,r)}f\,dx.
$$

Condition (2): There exists a constant $C^*>0$ such that for any ball $B(x,r)$, $x\in \mathrm{M}$, and $r>0$,
\be\label{cond2}
|B(x,2r)|\leq C^* |B(x,r)|.
\ee
The first condition is the Poincar\'e inequality. The second one is the doubling property of the measure.

 We aim to investigate the pointwise time analyticity of solutions to
 \be\label{PDE2}
\p_tu(t,x)-\mathrm{L}^\alpha u(t,x)=0,\ \alpha\in (0,2), \  (t,x)\in [0,1]\times\mathrm{M},
\ee	
where $\mathrm{L}^\alpha$ is defined as follows.
Let $\Delta$ be the Laplace operator on $\mathrm{M}$ generating a Markov semigroup $P_t$ which has a density $E(t,x;y)$ i.e.  the heat kernel of the standard heat equation on  $\mathrm{M}$. Consider the $\alpha$-stable subordination of $P_t$,
$$P_t^\alpha:=\int_{0}^\infty P_s\, \mu_t^\alpha(ds),\ t\geq 0,$$
where $\mu_t^\alpha$ is a probability measure on $[0,\infty)$ with the Laplace transform
$$
\int_0^\infty e^{-\lambda s}\,\mu_t^\alpha(ds)=e^{-t\lambda^\alpha},\ \lambda\geq 0.
$$
Then $\mathrm{L}^\alpha$ is the infinitesimal generator of $P^\alpha_t$.

In particular, we will also study the fractional heat kernel $p_\alpha(t,x;y)$ and its high order time derivatives $\p_t^kp_\alpha(t,x;y)$.

\begin{theorem}\label{theo03}
	Let $\mathrm{M}$ be a $d-$dimensional complete Riemannian manifold satisfying conditions \eqref{cond1} and \eqref{cond2} and $u=u(t,x)$ be a mild solution to equation \eqref{PDE2}, i.e.,
	\be\label{mild}
	u(t,x)=\int_{\mathrm{M}}p_\alpha(t,x;y)u(0,y)\,dy.
	\ee
	Assume that $u$ is of polynomial growth of order $(\alpha-\epsilon)$ at $t=0$, i.e., for a constant $C>0$,
	\be\label{grow3}
	|u(0,x)|\leq C(1+d(x,0)^{\alpha-\epsilon}),\ 0<\epsilon<\alpha,\  x \in \mathrm{M}.
	\ee
	Then for a constant $C>0$, it holds that
	\be\label{aim2}|\p_t^ku(t,x)|\leq \frac{C^{k+1} k^k}{t^{k-1}}\left(\frac{1+d(x,0)^{\alpha-\epsilon}}{t}+\frac{1}{t^{\epsilon/\alpha}}\right), \forall(t,x)\in (0,\infty)\times\mathrm{M},
\ee
	which implies that $u$ is time analytic  in $(0,\infty)\times\mathrm{M}$ with the radius of convergence independent of $x$.
\end{theorem}

We also obtain the time analyticity of the fractional heat kernel in the manifold setting.
\begin{theorem}\label{theo04}
	Let $\mathrm{M}$ be a $d-$dimensional complete Riemannian manifold satisfying conditions \eqref{cond1} and \eqref{cond2}.
	Then for any $t\in(0,\infty)$, there exist positive constants $C_1$ and $C_2$ such that the fractional heat kernel $p_\alpha(t,x;y)$ satisfies:
	\be\label{heatf}
	\frac{C_1t}{(d(x,y)^\alpha+t)|B(x,d(x,y)+t^{1/\alpha})|}\leq p_\alpha(t,x;y)\leq\frac{C_2t}{(d(x,y)^\alpha+t)|B(x,d(x,y)+t^{1/\alpha})|}.
	\ee
	Moreover, for any integer $k\geq 0$, there exists a constant $C>0$ such that
	\be\label{heat3}
	|\p_t^k p_\alpha(t,x;y)|\leq \frac{C^{k+1} k!}{t^{k-1}}\frac{1}{(d(x,y)^\alpha+t)|B(x,d(x,y)+t^{1/\alpha})|}.
	\ee
\end{theorem}
Here we remark that \eqref{heatf} is more or less known and our main contribution is \eqref{heat3}.
\begin{remark}
	It is an interesting question whether the uniqueness result still holds in the manifold case under the same growth condition. In the proof of Lemma \ref{le2}, we use \eqref{def} as an explicit formula for $\mathrm{L}_\alpha^{\kappa}$ in $\mathbb{R}^d$. However,  in $\mathrm{M}$, we do not have such a formula for $\mathrm{L}^\alpha$ in \eqref{PDE2}. Therefore, the proof in Lemma \ref{le2} does not work in this case.
\end{remark}

Now we give an outline of the rest of this paper. In Section \ref{2ndsec}, we investigate the pointwise time analyticity of a  solution of \eqref{PDE} in the setting of $\mathbb{R}^d$ and prove Theorem \ref{theo01}. In Section \ref{3rdsec}, by using the Fourier transform and contour integrals, we derive some estimates of the fractional heat kernel $p_\alpha(t,x)$, which implies Theorem \ref{theo02} and Corollary \ref{add2}. In Section \ref{4thsec}, we turn to the setting of a manifold and obtain similar results, Theorems \ref{theo03} and \ref{theo04}. In the proof, we use the subordination relation \eqref{connection} and the estimates for the standard heat kernel. Section \ref{5thsec} is devoted to some corollaries. One of them is about a necessary and sufficient condition for the solvability of the backward nonlocal parabolic equations. Another corollary gives a necessary and sufficient condition under which solutions to \eqref{PDE} or \eqref{PDE2} are time analytic at initial time $t=0$. Also for the nonlinear differential equation \eqref{nonlinear2} with power nonlinearity of order $p$, we prove that a solution $u=u(t,x)$ is time analytic in $t\in (0,1]$ if it is bounded in $[0,1]\times \mathrm{M}$ and $p$ is a positive integer.

Let us collect some frequently used notation.
\begin{itemize}
	\item If $x$ is in $\mathbb{R}^d$, then $|x|=\sqrt{\sum_{i=1}^d x_i^2}$ and $B_r(x)$ is a ball of radius $r$ centered at $x$.
	\item  In $\mathrm{M}$, $B(x, r)$ denotes the geodesic ball of radius $r$ centered at $x$ and
	$|B(x, r)|$  denotes its volume. We define $d(x, y)$ to be the geodesic distance of two points
	$x$, $y$ $\in\mathrm{M}$ and $0$ to be a reference point in $\mathrm{M}$.
	\item $p_\alpha(t,x;y)$ is the fractional heat kernel of equations \eqref{PDE}, \eqref{PDE3}, or \eqref{PDE2}, and $E(t,x;y)$ is the heat kernel of the usual heat equation.
\end{itemize}

Throughout this paper, the constant $C$ may differ from line to line.

\section{Nonlocal parabolic equations in \texorpdfstring{$\mathbb{R}^d$}{}}\label{2ndsec}

In this section, we  prove Theorem \ref{theo01} in the setting of $\mathbb{R}^d$. First, in Subsection \ref{analytic}, we prove that the fractional heat kernel $p_\alpha$ and the mild solution $u=u(t,x)$ to \eqref{PDE}, i.e. \eqref{mild}, are analytic in time. Next,  we prove that $u$ is the unique smooth solution
in Subsection \ref{unique}. Finally, we finish the proof of Theorem \ref{theo01} in Subsection \ref{last}. The proof is divided into several lemmas for easy reading.

\subsection{Time analyticity of the fractional heat kernel \texorpdfstring{$p_\alpha$}{} and mild solutions}\label{analytic}

\begin{lemma}\label{le1}
	Assume that $\kappa(\cdot,\cdot)$ satisfies \eqref{1stco} and \eqref{2ndco}.
	Then \eqref{goal} is true. Moreover,
	if the mild solution $$u=u(t,x)=\int_{\mathbb{R}^d} p_\alpha(t,x;y)u(0,y)\,dy$$ is of polynomial growth of order $\alpha-\epsilon$ as in \eqref{grow},
	then \eqref{goal1} holds.
\end{lemma}
\begin{proof}
From  \cite[(1.8), (1.14), and (1.10)]{[CZ]}, there exist constants $C_1$ and $C_2$ such that for any $t\in (0,1]$ and $x,y\in \mathbb{R}^d$,
\be\label{1st}
\frac{C_1 t}{\left(t^{1/\alpha}+|x-y|\right)^{d+\alpha}} \leq p_\alpha(t,x;y)\leq \frac{C_2 t}{\left(t^{1/\alpha}+|x-y|\right)^{d+\alpha}}
\ee
and
\be\label{2nd}
|\p_t p_\alpha(t,x;y)|\leq \frac{C_2}{\left(t^{1/\alpha}+|x-y|\right)^{d+\alpha}}.
\ee
Thus the conclusions of the lemma are true for $k=1$. Now we proceed by induction. For any integer $k>1$, we assume  that
$$|\p_t^{k-1} p_\alpha(t,x;y)|\leq \frac{ C^{k} (k-1)^{k-1}}{t^{k-2}}\frac{1}{\left(t^{1/\alpha}+|x-y|\right)^{d+\alpha}},\ t\in (0,1].$$
Without loss of generality, we may assume that $C_2\le C^{1/2}$.
Using the semigroup property and \eqref{2nd}, for any $t\in (0,1]$ and $\tau \in (0,t)$, we know that
$$
\p_t^k p_\alpha(t,x;y)=\int_{\mathbb{R}^d}\p_tp_{\alpha}(t-\tau,x;z)\p_{\tau}^{k-1}p_\alpha(\tau,z;y)\,dz.
$$
Therefore, by \eqref{2nd} and the inductive assumption, it holds that
\be\label{sep}
|\p_t^k p_\alpha(t,x;y)|\leq \frac{ C^{k+1/2} (k-1)^{k-1}}{\tau^{k-2}}\int_{\mathbb{R}^d}\frac{1}{\left((t-\tau)^{1/\alpha}+|x-z|\right)^{d+\alpha}}\frac{1}{\left(\tau^{1/\alpha}+|y-z|\right)^{d+\alpha}}\,dz.
\ee
Then for any $t\in (0,1]$, we take $\tau=\frac{(k-1)t}{k}$.

 On one hand, if $t>|x-y|^\alpha$, then we have
\be\label{case1}
\begin{aligned}
|\p_t^k p_\alpha(t,x;y)|&\leq \frac{ C^{k+1/2} (k-1)^{k-1}}{\tau^{k-2}}\frac{1}{\tau^{(d+\alpha)/\alpha}}\int_{\mathbb{R}^d}\frac{1}{\left((t-\tau)^{1/\alpha}+|x-z|\right)^{d+\alpha}}\,dz\\
&\leq \frac{ C^{k+3/4} (k-1)^{k-1}}{\tau^{k-2}}\frac{1}{\tau^{(d+\alpha)/\alpha}}\frac{1}{t-\tau}\\
&\leq\frac{C^{k+7/8} k^k}{t^{k-1}}\frac{1}{t^{(d+\alpha)/\alpha}}\leq  \frac{C^{k+1} k^k}{t^{k-1}}\frac{1}{\left(t^{1/\alpha}+|x-y|\right)^{d+\alpha}}
\end{aligned}
\ee
provided that $C$ is sufficiently large.

On the other hand, if $t<|x-y|^\alpha$, by \eqref{sep} and $$\mathbb{R}^d \subset \left\{z:|x-z| \geq \frac{|x-y|}{2}\right\}\cup \left\{z:|y-z| \geq \frac{|x-y|}{2}\right\},$$ we have
\be\label{sep3}
\begin{aligned}
	&|\p_t^k p_\alpha(t,x;y)|\\
	&\leq \frac{ C^{k+1/2} (k-1)^{k-1}}{\tau^{k-2}}\int_{\left\{z:|x-z| \geq |x-y|/2\right\}}\frac{1}{\left((t-\tau)^{1/\alpha}+|x-z|\right)^{d+\alpha}}\frac{1}{\left(\tau^{1/\alpha}+|y-z|\right)^{d+\alpha}}\,dz\\
	&\quad+\frac{ C^{k+1/2} (k-1)^{k-1}}{\tau^{k-2}}\int_{\left\{z:|y-z| \geq |x-y|/2\right\}}\frac{1}{\left((t-\tau)^{1/\alpha}+|x-z|\right)^{d+\alpha}}\frac{1}{\left(\tau^{1/\alpha}+|y-z|\right)^{d+\alpha}}\,dz\\
	&\leq \frac{ C^{k+1/2} (k-1)^{k-1}}{\tau^{k-2}}\frac{1}{\left((t-\tau)^{1/\alpha}+|x-y|/2\right)^{d+\alpha}}\int_{\left\{z:|x-z| \geq |x-y|/2\right\}}\frac{1}{\left(\tau^{1/\alpha}+|y-z|\right)^{d+\alpha}}\,dz\\
	&\quad+\frac{ C^{k+1/2} (k-1)^{k-1}}{\tau^{k-2}}\frac{1}{\left(\tau^{1/\alpha}+|x-y|/2\right)^{d+\alpha}}
\int_{\left\{z:|y-z| \geq |x-y|/2\right\}}\frac{1}{\left({(t-\tau)}^{1/\alpha}+|{x-z}|\right)^{d+\alpha}}\,dz\\
	&\leq \frac{ C^{k+3/4} (k-1)^{k-1}}{\tau^{k-2}}\frac{1}{\left((t-\tau)^{1/\alpha}+|x-y|/2\right)^{d+\alpha}}\frac{1}{\tau}\\
	&\quad +\frac{ C^{k+3/4} (k-1)^{k-1}}{\tau^{k-2}}\frac{1}{\left(\tau^{1/\alpha}+|x-y|/2\right)^{d+\alpha}}
	\frac{1}{t-\tau}.
\end{aligned}
\ee
Noting $\tau=\frac{(k-1)t}{k}$ and $t<|x-y|^\alpha$, by \eqref{sep3}, we can see that
\be\label{case2}
\begin{aligned}
&|\p_t^k p_\alpha(t,x;y)|\leq \frac{C^{k+7/8} k^k}{t^{k-1}}\frac{1}{|x-y|^{d+\alpha}}\leq \frac{C^{k+1} k^k}{t^{k-1}}\frac{1}{\left(t^{1/\alpha}+|x-y|\right)^{d+\alpha}}.
\end{aligned}
\ee
The combination of \eqref{case1} and \eqref{case2} completes the induction and gives \eqref{goal}.

Next we prove \eqref{goal1}. We claim that
\be\label{xx}
u(t,x)=\int_{\mathbb{R}^d}p_\alpha(t,x;y)u(0,y)\,dy,
\ee
the proof of which is postponed to the next subsection. Then we have
$$\p_t^k u(t,x)=\int_{\mathbb{R}^d}\p_t^k p_\alpha(t,x;y)u(0,y)\,dy.$$
This together with \eqref{goal} implies that
$$
\begin{aligned}
|\p_t^k u(t,x)|&\leq \int_{\mathbb{R}^d}|\p_t^k p_\alpha(t,x;y)||u(0,y)|\,dy\\
&\leq  \int_{\mathbb{R}^d}\frac{C^{k+1} k^k}{ t^{k-1}}\frac{1}{\left(t^{1/\alpha}+|x-y|\right)^{d+\alpha}}(1+|y|^{\alpha-\epsilon})\,dy\\
&\leq\int_{\mathbb{R}^d}\frac{C^{k+1} k^k}{ t^{k-1}}\frac{1}{\left(t^{1/\alpha}+|x-y|\right)^{d+\alpha}}(1+|x|^{\alpha-\epsilon}+|x-y|^{\alpha-\epsilon})\,dy\\
&\leq  \frac{C^{k+1} k^k}{ t^{k-1}}\Big(\frac {1+|x|^{\alpha-\epsilon}} t+\frac 1 {t^{\varepsilon/\alpha}}\Big),
\end{aligned}
$$
i.e., $u$ is time analytic when $t\in (0,1]$.
\end{proof}

\subsection{Uniqueness of solutions}\label{unique}

In this subsection, we prove that the mild solution
 $$u(t,x)=\int_{\mathbb{R}^d}p_\alpha(t,x;y)u(0,y)\,dy$$ in Theorem \ref{theo01} is unique among smooth solutions under the growth condition \eqref{grow}. This will imply  \eqref{xx}.
The proof is based on Propositions 3.4 and 3.5 in \cite{[DJZ]}, which we recall here for the reader's convenience. The idea is that once a solution is in $C^\gamma$ with a small $\gamma \in (0, 1)$, then it is in $C^\alpha$ with $\alpha\in[1,2)$.

 The first lemma is about the case when $\alpha\in (1,2)$.
\begin{lemma}[Proposition 3.4 in \cite{[DJZ]}]\label{33}
	Let $\omega_f(\cdot)$ be a modulus of continuity of a function $f=f(t,x)$ in $Q_{3/4}(1,x_0)$, that is
$$
	|f(t,x)-f(t',x')|\leq \omega_f(\max\{|x-x'|,|t-t'|^{1/\alpha}\}),\ \forall (t,x),(t',x')\in Q_{3/4}(1,x_0),
$$
	where $Q_r(t,x)=(t-r^\alpha,t)\times B_r(x)$.
	Assume that $u$ is a smooth solution to
	$$u_t(t,x)-\mathrm{L}_\alpha^{\kappa} u(t,x)=f(t,x),\ \alpha\in (1,2),\ (t,x)\in [0,1]\times \mathbb{R}^d,$$
	and $u\in C^\gamma([0,1]\times \mathbb{R}^d)$ for some $\gamma\in (0,1)$.
	Then it holds that
$$[u]^x_{\alpha;Q_{1/2}(1,x_0)}+[Du]^t_{(\alpha-1)/\alpha,Q_{1/2}(1,x_0)}+\|\p_tu\|_{L^\infty(Q_{1/2}(1,x_0))}\leq C\|u\|_{\gamma/\alpha,\gamma;[0,1]\times \mathbb{R}^d}+C\sum\limits_{k=1}^\infty\omega_f(2^{-k})$$
	for a constant $C>0$. Here
\begin{align*}
[u]^x_{\alpha;Q_{1/2}(1,x_0)}&:=\sup\limits_{t\in (1-(1/2)^\alpha,1)}[u(t,\cdot)]_{C^\alpha(B_{1/2}(x_0))},\\
[Du]^t_{(\alpha-1)/\alpha,Q_{1/2}(1,x_0)}&:=\sup\limits_{x\in B_{1/2}(x_0)}[Du(\cdot,x)]_{C^{(\alpha-1)/\alpha}((1-(1/2)^\alpha,1))},
\end{align*}
	and $\|u\|_{\gamma/\alpha,\gamma;[0,1]\times \mathbb{R}^d}$ is the $C^{\gamma/\alpha,\gamma}_{t,x}$ norm in $[0,1]\times \mathbb{R}^d$.
\end{lemma}

The second lemma is about the case when $\alpha=1$.
\begin{lemma}[Proposition 3.5 in \cite{[DJZ]}]\label{34}
Assume that $u$ is a smooth solution to
$$
u_t(t,x)-\mathrm{L}_\alpha^{\kappa} u(t,x)=f(t,x),\ \alpha=1,\ (t,x)\in [0,1]\times \mathbb{R}^d,
$$
and $u\in C^\gamma([0,1]\times \mathbb{R}^d)$ for some $\gamma\in (0,1)$.
Then it holds that
$$[Du]_{L^\infty(Q_{1/2}(1,x_0))}+\|\p_tu\|_{L^\infty(Q_{1/2}(1,x_0))}\leq C\|u\|_{\gamma,\gamma;[0,1]\times \mathbb{R}^d}+C\sum\limits_{k=1}^\infty\omega_f(2^{-k})$$
for a constant $C>0$.	
\end{lemma}

The proof of the uniqueness starts with the following lemma.
\begin{lemma}
Assume that $\kappa(\cdot,\cdot)$ satisfies \eqref{1stco} and \eqref{2ndco}. For equation \eqref{PDE}, suppose that a smooth solution $u=u(t,x)$ is of polynomial growth of order $\alpha-\epsilon$, i.e.,
\be\label{grow2}
|u(t,x)|\leq C_1 \left(1+|x|^{\alpha-\epsilon}\right),\forall (t,x)\in [0,1]\times \mathbb{R}^d,\ \alpha\in [1,2),\ \epsilon\in (0,\alpha).
\ee
Then for a constant $C>0$ and for any $x_0 \in {R}^d$, it holds that
\be\label{dgrow} [u]^x_{1;Q_{1/2}(1,x_0)} \leq C\left(1+|x_0|^{\alpha-\epsilon}\right),\ \epsilon>0,
\ee
 where
 $$
 [u]^x_{1;Q_{1/2}(1,x_0)}:=\sup\limits_{t\in(1-(1/2)^\alpha,1)}\|u(t,\cdot)\|_{\text{Lip} (B_{1/2}(x_0))}$$ and $\text{Lip}$ means the Lipschitz norm.
\end{lemma}
\begin{proof}
From Proposition 2.4 of \cite{[DZ1]} or Theorem 7.1 of \cite{[SS]}, there is a small constant $\gamma\in(0,1)$ such that
\be\label{fi1}[u]_{\gamma/\alpha,\gamma;Q_{7/8}(1,0)}\leq C\|u\|_{L^\infty((0,1);L_1(\omega_\alpha))},\ee
where $\omega_\alpha=\frac{1}{1+|x|^{d+\alpha}}$ and $$\|u\|_{L^\infty((0,1);L_1(\omega_\alpha))}=\sup\limits_{t\in(0,1)}\int_{\mathbb{R}^d}\frac{|u(t,x)|}{1+|x|^{d+\alpha}}\,dx.$$
By \eqref{fi1}, the growth condition \eqref{grow2}, and the space translation $x\to x+x_0$ for any $x_0\in \mathbb{R}^d$, we have
\be\label{hold}
\begin{aligned}
&[u]_{\gamma/\alpha,\gamma;Q_{7/8}(1,x_0)} \leq C\sup\limits_{t\in(0,1)}\int_{\mathbb{R}^d}\frac{|u(t,x+x_0)|}{1+|x|^{d+\alpha}}\,dx\\
&\leq C\int_{\mathbb{R}^d}\frac{(1+|x|^{\alpha-\epsilon}+|x_0|^{\alpha-\epsilon})}{1+|x|^{d+\alpha}}\,dx\leq C(1+|x_0|^{\alpha-\epsilon}).
\end{aligned}
\ee
The next step is to prove
\be\label{hol}
\begin{aligned}
	&\quad[u]^x_{\alpha;Q_{5/8}(1,x_0)}\leq C(1+|x_0|)^{\alpha-\epsilon}.
\end{aligned}
\ee

We modify the proof of Theorem 1.1 of \cite{[DJZ]}.

 Take a cut-off function $\eta=\eta(t,x)\in C_0^\infty(Q_{7/8}(1,x_0))$ satisfying $\eta=1$ in $Q_{5/6}(1,x_0)$ and $\|\p_t^jD^i \eta\|_{L^\infty}\leq C $ when $i\in\{0,1,2\}$ and $j\in\{0,1\}$.

 Let $(t,x), (t',x')$ be two points in $ Q_{3/4}(1,x_0)$ and let $v(t,x):=u(t,x)\eta(t,x)$. Then in $Q_{3/4}(1,x_0)$,
\be\label{fi}
\p_t v=\eta\p_tu+\p_t\eta u=\eta L_{\alpha}^\kappa u+\p_t\eta u= L_{\alpha}^\kappa v+h+\p_t\eta u,
\ee
where
$$h=\eta L_{\alpha}^\kappa u-  L_{\alpha}^\kappa v=p.v. \int_{R^d}\frac{\xi(t,x,y)\kappa(x,y)}{|y|^{d+\alpha}}\,dy$$
and
\be\label{xi}
\xi(t,x,y)=u(t,x+y)(\eta(t,x)-\eta(t,x+y)).
\ee

We are going to apply Lemma \ref{33} or Lemma \ref{34} to \eqref{fi} in $Q_{3/4}(1,x_0)$ and obtain corresponding estimates \eqref{hol} in $Q_{5/8}(1,x_0)$.
To this end, we only need to estimate the H\"older semi-norm of $h$ in $Q_{3/4}(1,x_0)$.

First, when $|y|\leq 5/6-3/4=1/12$, by \eqref{xi}, we have
\be\label{*}
\xi(t,x,y)=\xi(t',x',y)=0.
\ee
By the assumptions on $\eta$ and \eqref{xi},  it holds that
\begin{equation}
            \label{eq5.17}
|\xi(t',x',y)|\leq\left\{\begin{array}{ll}
C|u(t',x'+y)|, & |y|\geq 1 \\
C |u(t',x'+y)||y|, &1/12<|y|<1.
\end{array}\right.
\end{equation}
Now by the triangle inequality, we deduce that
 \be\label{split}
 \begin{aligned}
 & |h(t,x)-h(t',x')|\\
 &\leq \underbrace{\int_{R^d}\frac{|(\xi(t,x,y)-\xi(t',x',y))\kappa(x,y)|}{|y|^{d+\alpha}}\,dy}_{\uppercase\expandafter{\romannumeral1}}\\
&\quad +\underbrace{\int_{R^d}\frac{|\xi(t',x',y)(\kappa(x',y)-\kappa(x,y))|}{|y|^{d+\alpha}}\,dy}_{\uppercase\expandafter{\romannumeral2}}.
\end{aligned}
\ee
By using \eqref{2ndco}, \eqref{grow2}, \eqref{*}, and \eqref{eq5.17}, we have
\be\label{II}
\begin{aligned}
\uppercase\expandafter{\romannumeral2}
&\leq  \int_{|y|\in(1/12,1)}\frac{C  |u(t',x'+y)||y|\kappa_2|x-x'|^\beta}{|y|^{d+\alpha}}\,dy+\int_{|y|>1}\frac{C|u(t',x'+y)|}{|y|^{d+\alpha}}\kappa_2|x-x'|^\beta \,dy\\
&\leq \int_{|y|\in(1/12,1)}\frac{C  (1+|x_0|^{\alpha-\epsilon}+|y|^{\alpha-\epsilon})|x-x'|^\beta}{|y|^{d+\alpha-1}}\,dy\\
&\quad+\int_{|y|>1}\frac{C(1+|x_0|^{\alpha-\epsilon}+|y|^{\alpha-\epsilon})}{|y|^{d+\alpha}}|x-x'|^\beta \,dy \leq C(1+|x_0|^{\alpha-\epsilon})|x-x'|^\beta.
\end{aligned}
\ee
Now we estimate $\uppercase\expandafter{\romannumeral1}$.
When $1/12 \leq|y|<2,$ by the fundamental theorem of calculus, we have
$$
\begin{aligned}
\xi(t,x,y)-\xi(t',x',y)=-y\int_0^1 \left(u(t,x+y)D\eta(t,x+sy)-u(t',x'+y)D\eta(t',x'+sy)\right)\,ds.
\end{aligned}
$$
Therefore, by \eqref{grow2}, \eqref{hold}, and the triangle inequality, it holds that
\be\label{fi2}
\begin{aligned}
& \left|\xi(t, x, y)-\xi\left(t^{\prime}, x^{\prime}, y\right)\right|\\
&\leq|y|\int_0^1\left|u(t,x+y)-u(t',x'+y)\right|\left|D\eta(t',x'+sy)\right|\,ds\\
&\quad+|y|\int_0^1|u(t,x+y)|\left|D\eta(t,x+sy)-D\eta(t',x'+sy)\right|\,ds\\
& \leq C|y|\left| u(t,x+y)-u(t',x'+y)\right|+C|y||u(t,x+y)|\left(|x-x^{\prime}|+|t-t^{\prime}|\right)\\
&\leq C|y|(1+|x_0|^{\alpha-\epsilon})\left(|x-x^{\prime}|^\gamma+|t-t^{\prime}|^{\gamma/\alpha}\right)+C|y|(1+|x_0|^{\alpha-\epsilon})\left(|x-x^{\prime}|+|t-t^{\prime}|\right).
\end{aligned}
\ee
When $|y| \geq 2,$ by \eqref{xi} and \eqref{hold}, we have
\be\label{fi3}
\begin{aligned}
&\left|\xi(t, x, y)-\xi\left(t^{\prime}, x^{\prime}, y\right)\right|=\left|u(t, x+y)-u\left(t^{\prime}, x^{\prime}+y\right)\right|\\
&\le C(1+|x_0|^{\alpha-\epsilon}+|y|^{\alpha-\epsilon})\left(|x-x^{\prime}|^\gamma+|t-t^{\prime}|^{\gamma/\alpha}\right).
\end{aligned}
\ee
Thus, by \eqref{1stco}, \eqref{fi2}, \eqref{fi3}, and \eqref{*}, we infer that
\be\label{I}
\begin{aligned}
\uppercase\expandafter{\romannumeral1}&\leq \int_{|y|\in(1/12,2)}\frac{C|y|(1+|x_0|^{\alpha-\epsilon})\left(|x-x^{\prime}|^\gamma+|t-t^{\prime}|^{\gamma/\alpha}\right)}{|y|^{d+\alpha}}\,dy\\
&\quad+\int_{|y|\in(1/12,2)}\frac{C|y|(1+|x_0|^{\alpha-\epsilon})\left(|x-x^{\prime}|+|t-t^{\prime}|\right)}{|y|^{d+\alpha}}\,dy\\
&\quad+\int_{|y|>2}\frac{C(1+|x_0|^{\alpha-\epsilon}+|y|^{\alpha-\epsilon})\left(|x-x^{\prime}|^\gamma+|t-t^{\prime}|^{\gamma/\alpha}\right)}{|y|^{d+\alpha}}\,dy\\
&\leq C(1+|x_0|^{\alpha-\epsilon})\left(|x-x^{\prime}|^\gamma+|t-t^{\prime}|^{\gamma/\alpha}\right).
\end{aligned}
\ee
Plugging \eqref{II} and \eqref{I} into \eqref{split}, we deduce that
$$|h(t,x)-h(t',x')|\leq C(1+|x_0|^{\alpha-\epsilon})\left(|x-x^{\prime}|^{\gamma'}+|t-t^{\prime}|^{\gamma^\prime/\alpha}\right),$$
where $\gamma'=\min\{\gamma,\beta\}$, which implies that we can take the modulus of continuity as
$$
\omega_h(r) = C(1+|x_0|^{\alpha-\epsilon}) r^{\gamma'}
$$
for any $r\in (0,1)$. According to Lemma \ref{33}, it follows that
\be\label{ite}
\sum_{k=1}^\infty \omega_h\left(\frac{3}{2^{k+1}}\right)\leq \sum_{k=1}^\infty C(1+|x_0|^{\alpha-\epsilon})\left(\frac{3}{2^{k+1}}\right)^{\gamma'}\leq C(1+|x_0|^{\alpha-\epsilon}).
\ee

Now we consider two cases.

\textbf{Case (1):} $\boldsymbol{\alpha\in (1,2)}$. In this case, we apply Lemma \ref{33} to \eqref{fi} in $Q_{3/4}(1,x_0)$ with a scaling argument. From \eqref{hold} and \eqref{ite}, we have
$$
\begin{aligned}
&[v]^x_{\alpha;Q_{5/8}(1,x_0)}\leq C\|v\|_{L^\infty([0,1]\times \mathbb{R}^d)}+C[v]_{\gamma/\alpha,\gamma;[0,1]\times \mathbb{R}^d}+C\sum_{k=1}^\infty \omega_h\left(\frac{3}{2^{k+1}}\right)\\
&\leq C\|u\|_{L^\infty(Q_{7/8}(1,x_0))}+C[u]_{\gamma/\alpha,\gamma;Q_{7/8}(1,x_0)}+C(1+|x_0|^{\alpha-\epsilon})\leq C(1+|x_0|^{\alpha-\epsilon}),
\end{aligned}
$$
by noting that $v=0$ outside of $Q_{7/8}(1,x_0)$.
Because $\eta=1$ in $Q_{5/8}(1,x_0)$, we get \eqref{hol} immediately.

\textbf{Case (2):} $\boldsymbol{\alpha=1}$. In this case, we apply Lemma \ref{34}  with a scaling argument. Using  \eqref{hold} and \eqref{ite}, we have
$$
\begin{aligned}
\|Dv\|_{L^\infty(Q_{5/8}(1,x_0))}&\leq C\|v\|_{L^\infty([0,1]\times \mathbb{R}^d)}+C[v]_{\gamma,\gamma;[0,1]\times \mathbb{R}^d}+C\sum_{k=1}^\infty \omega_h\left(\frac{3}{2^{k+1}}\right)\\
&\leq C\|u\|_{L^\infty(Q_{7/8}(1,x_0))}+C[u]_{\gamma,\gamma;Q_{7/8}(1,x_0)}+C(1+|x_0|^{\alpha-\epsilon})\leq C(1+|x_0|^{\alpha-\epsilon}),
\end{aligned}
$$
which implies \eqref{hol} again.

Finally, by the interpolation inequality, \eqref{hol}, and \eqref{grow2}, we arrive at
$$[u]^x_{1;Q_{1/2}(1,x_0)}\leq C[u]^x_{\alpha;Q_{5/8}(1,x_0)}+C\|u\|_{L^\infty(Q_{5/8}(1,x_0))} \leq C(1+|x_0|)^{\alpha-\epsilon},$$
which finishes the proof.
\end{proof}

Now we are ready to prove the uniqueness part of the theorem, which is stated as follows.
\begin{lemma}\label{le2}
Assume that $\kappa(\cdot,\cdot)$ satisfies \eqref{1stco} and \eqref{2ndco}. Then there is an unique smooth solution $u=u(t,x)$ to \eqref{PDE} satisfying the initial data $u(0,\cdot)$ and the polynomial growth condition \eqref{grow}, which is given by
$$u(t,x)=\int_{\mathbb{R}^d} p_\alpha(t,x;y)u(0,y)\,dy ,\, \forall(t,x)\in (0,1]\times\mathbb{R}^d.$$
\end{lemma}
\begin{proof} By linearity, we just need to prove that if a smooth solution $u$ satisfies \eqref{grow} and $u(0,x)=0$, then $u\equiv 0$.

Fix $(t_0,x_0)\in (0,1]\times \mathbb{R}^d$.
By shifting the coordinates, we may assume $x_0=0$ and it suffices to prove $u(t_0,0)=0$.
Now let $L^*=({L}_{\alpha}^\kappa)^*$ be the adjoint operator of $L_{\alpha}^\kappa$ and let $p_\alpha^*(t,x;s,y)$ be the heat kernel of $L^*$, which by definition, satisfies
\be\label{adj1}
\left\{\begin{array}{l}
\partial_{t} p_\alpha^*(t,x;s,y)-L^* p_\alpha^*(t,x;s,y)=0, \ t>s\ \text{and}\
 x,y \in \mathbb{R}^d  \\
p_\alpha^*(s,x;s,y)=\delta(x,y).
\end{array}\right.
\ee
Because the heat kernels of $L_\alpha^\kappa$ and $L^*$ are independent of time, we have
\be\label{adj2}
p_\alpha(t,x;s,y)=p_\alpha(t-s,x;0,y),\quad
p_\alpha^*(t,x;s,y)=p_\alpha^*(t-s,x;0,y).
\ee
It is also known that
\be\label{adj}
p_\alpha(t,x;s,y)=p_\alpha^*(t,y;s,x),\ t\geq s,
\ee
which can be seen as follows.
For any $t_0,s_0\in(0,1)$ with $s_0\leq t_0$, using \eqref{adj1} and \eqref{adj2}, we have
$$
\begin{aligned}
& \int_{s_0}^{t_0}\int_{\mathbb{R}^d}{L_\alpha^\kappa} p_\alpha(t,z;s_0,y)p_\alpha^*(t_0,z;t,x)\,dzdt\\
&=\int_{s_0}^{t_0}\int_{\mathbb{R}^d}L_\alpha^\kappa p_\alpha(t-s_0,z;0,y)p_\alpha^*(t_0-t,z;0,x)\,dzdt\\
&=\int_{s_0}^{t_0}\int_{\mathbb{R}^d}\p_t p_\alpha(t-s_0,z;0,y)p_\alpha^*(t_0-t,z;0,x)\,dzdt\\
&=p_\alpha(t_0-s_0,x;0,y)-p_\alpha^*(t_0-s_0,y;0,x)+\int_{s_0}^{t_0}p_\alpha(t-s_0,z;0,y)\p_t p_\alpha^*(t_0-t,z;0,x)\,dzdt.
\end{aligned}
$$
By the definition of the adjoint operator, \eqref{adj1}, and \eqref{adj2}, we reach \eqref{adj}. The
integrations above are justified due to known decay estimates of $p_\alpha$ and $p^*_\alpha$.

Then we take a cut-off function $\eta=\eta(x)\in C_c^\infty(B_2(0))$ such that  for a constant $C$,
\be\label{cut1}
\eta=1\ \text{in}\ B_1(0)\quad \text{and}\quad |D\eta|+|D^2 \eta|\leq C.
\ee
We test \eqref{PDE} with $p_\alpha^*(t_0-t,x;0,0)\eta(x/R)$ and use \eqref{adj1} to get that
$$
\begin{aligned}
0&=\int_{0}^{t_0}\int_{\mathbb{R}^d}u_t(t,x)p_\alpha^*(t_0-t,x;0,0)\eta(x/R)\,dxdt\\
&\quad-\int_{0}^{t_0}\int_{\mathbb{R}^d} L_\alpha^\kappa u(t,x) p_\alpha^*(t_0-t,x;0,0)\eta(x/R)\,dxdt\\
&=u(t_0,0)
+\int_{0}^{t_0}\int_{\mathbb{R}^d}u(t,x) (\p_tp_\alpha^*)(t_0-t,x;0,0)\eta(x/R)\,dxdt\\
&\quad-\int_{0}^{t_0}\int_{\mathbb{R}^d} L_\alpha^\kappa u(t,x) p_\alpha^*(t_0-t,x;0,0)\eta(x/R)\,dxdt.
\end{aligned}
$$
Therefore, using \eqref{adj1} and the definition of the adjoint operator, we infer that
\be\label{u}
\begin{aligned}
& u(t_0,0)\\
&=\int_{0}^{t_0}\int_{\mathbb{R}^d} L_\alpha^\kappa (u(t,x))(p_\alpha^*(t_0-t,x;0,0)\eta(x/R))-p_\alpha^*(t_0-t,x;0,0)L_\alpha^\kappa\left(u(t,x)\eta(x/R)\right)\,dxdt\\
&=p.v. \int_{0}^{t_0}\int_{\mathbb{R}^d}\int_{\mathbb{R}^d}\frac{u(t,x+z)p_\alpha^*(t_0-t,x;0,0)\left(\eta(x/R)-\eta((x+z)/R)\right)\kappa(x,z)}{|z|^{d+\alpha}}\,dzdxdt\\
&=p.v.\underbrace{\int_{0}^{t_0}\int_{\mathbb{R}^d}\int_{\mathbb{R}^d}\frac{u(t,y)p_\alpha^*(t_0-t,x;0,0)(\eta(x/R)-\eta(y/R))\kappa(x,y-x)}{|x-y|^{d+\alpha}}\,dydxdt}_{J_1},
\end{aligned}
\ee
where we took $z=y-x$ in the last step. In the sequel, we omit $p.v.$ when there is no confusion.

Next, we aim to show that $J_1 \to 0$ as $R \to \infty$, treating the cases $\alpha<1$ and $\alpha \ge 1$ separately.

\textbf{Case (1):} $\boldsymbol{\alpha<1}$. This case is simpler since the singularity in the integrand is weaker. Using \eqref{grow}, \eqref{1stco}, \eqref{adj}, and \eqref{cut1}, we have
$$
\begin{aligned}
J_1&=\int_{0}^{t_0}\int_{\mathbb{R}^d}\int_{\mathbb{R}^d\backslash B_R(x)}\frac{u(t,y)p_\alpha^*(t_0-t,x;0,0)(\eta(x/R)-\eta(y/R))\kappa(x,y-x)}{|x-y|^{d+\alpha}}\,dydxdt\\
&\quad+\int_{0}^{t_0}\int_{\mathbb{R}^d}\int_{B_R(x)}\frac{u(t,y)p_\alpha^*(t_0-t,x;0,0)(\eta(x/R)-\eta(y/R))\kappa(x,y-x)}{|x-y|^{d+\alpha}}\,dydxdt\\
&\leq C\int_{0}^{t_0}\int_{\mathbb{R}^d}\int_{\mathbb{R}^d\backslash B_R(x)}\frac{p_\alpha(t_0-t,0;0,x)}{|x-y|^{d+\alpha}}(1+|y|^{\alpha-\epsilon})\,dydxdt\\
&\quad+\frac{C}{R}\int_{0}^{t_0}\int_{\mathbb{R}^d}\int_{B_R(x)}\frac{p_\alpha(t_0-t,0;0,x)}{|x-y|^{d+\alpha-1}}(1+|y|^{\alpha-\epsilon})\,dydxdt\\
&\leq C\int_{0}^{t_0}\int_{\mathbb{R}^d}\int_{\mathbb{R}^d\backslash B_R(x)}\frac{p_\alpha(t_0-t,0;0,x)}{|x-y|^{d+\alpha}}(1+|x|^{\alpha-\epsilon}+|x-y|^{\alpha-\epsilon})\,dydxdt\\
&\quad+\frac{C}{R}\int_{0}^{t_0}\int_{\mathbb{R}^d}\int_{B_R(x)}\frac{p_\alpha(t_0-t,0;0,x)}{|x-y|^{d+\alpha-1}}(1+|x|^{\alpha-\epsilon}+|x-y|^{\alpha-\epsilon})\,dydxdt\\
&\leq C\int_{0}^{t_0}\int_{\mathbb{R}^d}p_\alpha(t_0-t,0;0,x)\left(\frac{1}{R^\epsilon}+\frac{1+|x|^{\alpha-\epsilon}}{R^\alpha}\right)\,dxdt \to 0\  \text{as}\  R\to \infty,
\end{aligned}
$$
where for the last step, we used \eqref{1st} and
\be\label{int2}
\begin{aligned}
& \int_{\mathbb{R}^d}p_\alpha(t_0-t,0;0,x)(1+|x|^{\alpha-\epsilon})\,dx\\
&\leq\int_{\mathbb{R}^d}\frac{C(t_0-t)}{\left((t_0-t)^{1/\alpha}+|x|\right)^{d+\alpha}}(1+|x|^{\alpha-\epsilon})\,dx\leq C\left(1+(t_0-t)^{1-\epsilon/\alpha}\right).
\end{aligned}
\ee

\textbf{Case (2):} $\boldsymbol{\alpha\geq 1}$. In this case, by the substitution $z \to -z$ in the second line of \eqref{u}, we have
$$
\begin{aligned}
J_1&=\int_{0}^{t_0}\int_{\mathbb{R}^d}\int_{\mathbb{R}^d}\frac{u(t,x-z)p_\alpha^*(t_0-t,x;0,0)\left(\eta(x/R)-\eta((x-z)/R)\right)\kappa(x,z)}{|z|^{d+\alpha}}\,dzdxdt.
\end{aligned}
$$
 where we used $\kappa(x,z)=\kappa(x,-z)$ in the last equation.
Then by
$$
\begin{aligned}
& u(t,x+z)\left(\eta\left(\frac{x}{R}\right)-\eta\left(\frac{x+z}{R}\right)\right)+u(t,x-z)\left(\eta\left(\frac{x}{R}\right)-\eta\left(\frac{x-z}{R}\right)\right)\\
&=(u(t,x-z)-u(t,x+z))
\left(\eta\left(\frac{x}{R}\right)-\eta\left(\frac{x-z}{R}\right)\right)\\
&\quad -u(t,x+z)\left(\eta\left(\frac{x+z}{R}\right)-2\eta\left(\frac{x}{R}\right)+\eta\left(\frac{x-z}{R}\right)\right),
\end{aligned}
$$
we can write
$$
\begin{aligned}
&J_1=\\
&\underbrace{\frac{1}{2}\int_{0}^{t_0}\int_{\mathbb{R}^d}\int_{\mathbb{R}^d}
\frac{(u(t,x-z)-u(t,x+z))\left(\eta(\frac{x}{R})-\eta(\frac{x-z}{R})\right)\kappa(x,z)p_\alpha^*(t_0-t,x;0,0)}{|z|^{d+\alpha}}\,dzdxdt}_{J_2}\\
&\quad+\underbrace{\frac{1}{2}\int_{0}^{t_0}\int_{\mathbb{R}^d}\int_{\mathbb{R}^d}\frac{-u(t,x+z)\left(\eta(\frac{x+z}{R})-2\eta(\frac{x}{R})+\eta(\frac{x-z}{R})\right)
\kappa(x,z)p_\alpha^*(t_0-t,x;0,0)}{|z|^{d+\alpha}}\,dzdxdt}_{J_3}.
\end{aligned}
$$
For the term $J_3$, by \eqref{grow}, \eqref{adj}, and \eqref{cut1}, we deduce
$$
\begin{aligned}
|J_3|&\leq C\int_{0}^{t_0}\int_{\mathbb{R}^d}\int_{\mathbb{R}^d\backslash B_R(0)}\frac{p_\alpha(t_0-t,0;0,x)}{|z|^{d+\alpha}}(1+|x|^{\alpha-\epsilon}+|z|^{\alpha-\epsilon})\,dzdxdt\\
&\quad+\frac{C}{R^2}\int_{0}^{t_0}\int_{\mathbb{R}^d}\int_{B_R(0)}\frac{p_\alpha(t_0-t,0;0,x)}{|z|^{d+\alpha-2}}(1+|x|^{\alpha-\epsilon}+|z|^{\alpha-\epsilon})\,dzdxdt\\
&\leq C\int_{0}^{t_0}\int_{\mathbb{R}^d}p_\alpha(t_0-t,0;0,x)\left(\frac{1}{R^\epsilon}
+\frac{1+|x|^{\alpha-\epsilon}}{R^\alpha}\right)\,dxdt \to 0\  \text{as}\  R\to \infty,
\end{aligned}
$$
where we used \eqref{int2} in the last step.

Finally, we estimate $J_2$. When $\alpha>1$,  by \eqref{grow}, \eqref{dgrow},  and \eqref{int2}, we have
$$
\begin{aligned}
|J_2|&\leq  C\int_{0}^{t_0}\int_{\mathbb{R}^d}\int_{\mathbb{R}^d\backslash B_R(0)}\frac{p_\alpha(t_0-t,0;0,x)}{|z|^{d+\alpha}}(1+|x|^{\alpha-\epsilon}+|z|^{\alpha-\epsilon})\,dzdxdt\\
&\quad+\frac{C}{R^2}\int_{0}^{t_0}\int_{\mathbb{R}^d}\int_{B_R(0)}\frac{p_\alpha(t_0-t,0;0,x)}{|z|^{d+\alpha-2}}(1+|x|^{\alpha-\epsilon})\,dzdxdt\\
&\quad+\underbrace{\frac{C}{R}\int_{0}^{t_0}\int_{\mathbb{R}^d}\int_{B_R(0)\backslash B_R(0)}\frac{p_\alpha(t_0-t,0;0,x)}{|z|^{d+\alpha-1}}(1+|x|^{\alpha-\epsilon}+|z|^{\alpha-\epsilon})\,dzdxdt}_{J_4}\\
&\leq C\int_{0}^{t_0}\int_{\mathbb{R}^d}p_\alpha(t_0-t,0;0,x)\left(\frac{1}{R^\epsilon}+\frac{1+|x|^{\alpha-\epsilon}}{R^\alpha}\right)\,dxdt\\
&\quad+\frac{C}{R}\int_{0}^{t_0}\int_{\mathbb{R}^d}p_\alpha(t_0-t,0;0,x)\left((1-R^{1-\alpha})(1+|x|^{\alpha-\epsilon})+(R^{1-\epsilon}-1)\right)\,dxdt\\
& \to 0\quad  \text{as}\  R\to \infty.
\end{aligned}
$$
When $\alpha=1$, we only need to estimate $J_4$ slightly differently. By \eqref{int2},
$$J_4\leq \frac{C}{R}\int_{0}^{t_0}\int_{\mathbb{R}^d}p_1(t_0-t,0;0,x)\left(\ln(R)(1+|x|^{1-\epsilon})+(R^{1-\epsilon}-1)\right)\,dx\to 0\quad  \text{as}\  R\to \infty.$$
Combining these two cases and plugging into \eqref{u}, we get $u(t_0,0)=0$, which finishes the proof.
\end{proof}

\subsection{Completion of proof of Theorem \ref{theo01}}\label{last}
\begin{proof}
We have proved part (a) and (b) of Theorem \ref{theo01} in Lemmas \ref{le1} and \ref{le2}. Thus it remains to show part (c). First we fix a number $R \geq 1$ and let $x \in B_R(0)$, $t \in[1-\delta, 1]$ for some small $\delta>0$.  For any positive integer $j$, Taylor's theorem implies that
\begin{equation}\label{taylor}
u(t,x)-\sum_{i=0}^{j-1} \partial_{t}^i u(1,x) \frac{(t-1)^{i}}{i !}=\frac{(t-1)^{j}}{j !} \partial_{t}^{j} u(s,x),
\end{equation}
where $s=s(x, t, j) \in[t, 1]$. By \eqref{goal1}, for sufficiently small $\delta>0,$ the right-hand side of
\eqref{taylor} converges to 0 uniformly with respect to $x\in B_R(0)$ as $j \rightarrow \infty$. Hence,
$$
u(t,x)=\sum_{j=0}^{\infty} \partial_{t}^{j} u(1,x) \frac{(t-1)^{j}}{j !}
$$
i.e., $u$ is analytic in time with radius $\delta$. Denote $a_{j}=a_{j}(x)=\partial_{t}^{j} u(1,x)$. By \eqref{goal1} again, we
have
$$
\partial_{t} u(t,x)=\sum_{j=0}^{\infty} a_{j+1}(x) \frac{(t-1)^{j}}{j !}\quad \text {and}\quad \mathrm{L}_\alpha^{\kappa} u(t,x)=\sum_{j=0}^{\infty} \mathrm{L}_\alpha^{\kappa} a_{j}(x) \frac{(t-1)^{j}}{j !},
$$
where both series converge uniformly  with respect to $(t,x)\in [1-\delta,1]\times B_R(0)$. Since $u$ is a solution of \eqref{PDE}, this implies that $\mathrm{L}_\alpha^{\kappa} a_{j}(x)=a_{j+1}(x)$
with
$$
\left|a_{j}(x)\right| \leq C^{j+1}j^j(1+|x|^{\alpha-\epsilon}).
$$
This completes the proof
of Theorem \ref{theo01}.
\end{proof}

\section{Fractional heat kernel estimates on \texorpdfstring{$\mathbb{R}^d$}{}}\label{3rdsec}
In this section, we estimate the time and space derivatives of the fractional heat kernel $p_\alpha(t,x)$ for \eqref{PDE3}. The main tools are the Fourier transform and contour integrals.
We first state and prove the following lemma, which is needed for the proof of Theorem \ref{theo02} and Corollary \ref{add2}.
\begin{lemma}\label{main}
(a) If $\alpha>0$, $\beta\geq 0$, and $t\geq 0$,
		there exist constants $C$, $C_1$, and $C_2$ such that
		\be\label{any}
		\big|\int_{\mathbb{R}^d} e^{-t|\xi|^\alpha}e^{i \xi x}|\xi|^{\beta}d\xi \big|\leq \min\left\{\frac{C_1 C_2^{\beta} \beta^{\beta}}{|x|^{\beta+d}},\frac{C}{t^{(\beta+d)/\alpha}}\Gamma\left(\frac{\beta+d}{\alpha}\right)\right\},
		\ee
		where $\Gamma$ is the gamma function.
		
(b) Let $\boldsymbol{\beta}=(\beta_1,\beta_2,\ldots,\beta_d)$ where $\beta_j$ is a nonnegative integer with $j\in\{1,2,\ldots,d\}$, then we have
		\be\label{int}
		\big|\int_{\mathbb{R}^d} e^{-t|\xi|^\alpha}e^{i \xi x}\xi^{\boldsymbol{\beta}}d\xi \big|\leq \min\left\{\frac{C_1 C_2^{\alpha+|\boldsymbol{\beta}|} (\alpha+|\boldsymbol{\beta}|)^{\alpha+|\boldsymbol{\beta}|} t}{|x|^{\alpha+|\boldsymbol{\beta}|+d}},\frac{C}{t^{(|\boldsymbol{\beta}|+d)/\alpha}}\Gamma\left(\frac{|\boldsymbol{\beta}|+d}{\alpha}\right)\right\},
		\ee
where $\xi^{\boldsymbol{\beta}}=\xi_1^{\beta_1}\xi_2^{\beta_2}\cdots \xi_d^{\beta_d}$ and $|\boldsymbol{\beta}|:=\sum\limits_{k=1}^d \beta_k$.	
\end{lemma}
\begin{remark}
	When $t=0$, the integrals in \eqref{any} and \eqref{int} can be understood as the limit as $t\searrow 0$.
\end{remark}
\begin{proof}[Proof of Lemma \ref{main}]
The bound $\frac{C}{t^{(\beta+d)/\alpha}}\Gamma\left(\frac{\beta+d}{\alpha}\right)$ on the right-hand side of \eqref{any}
is easily obtained as follows
$$\big|\int_{\mathbb{R}^d} e^{-t|\xi|^\alpha}e^{i \xi x}|\xi|^{\beta}d\xi\big|\leq \int_{\mathbb{R}^d} e^{-t|\xi|^\alpha}|\xi|^{\beta}d\xi=\frac{C}{t^{(\beta+d)/\alpha}}\Gamma\left(\frac{\beta+d}{\alpha}\right).$$
Similarly, the bound $\frac{C}{t^{(|\boldsymbol{\beta}|+d)/\alpha}}\Gamma\left(\frac{|\boldsymbol{\beta}|+d}{\alpha}\right)$ on the right-hand side of \eqref{int} holds because 	
$$\big|\int_{\mathbb{R}^d} e^{-t|\xi|^\alpha}e^{i \xi x}\xi^{\boldsymbol{\beta}}d\xi \big|\leq \int_{\mathbb{R}^d} e^{-t|\xi|^\alpha}|\xi|^{\boldsymbol{|\beta|}}d\xi = \frac{C}{t^{(|\boldsymbol{\beta}|+d)/\alpha}}\Gamma\left(\frac{|\boldsymbol{\beta}|+d}{\alpha}\right).$$
	
We shall use the technique of contour integrals to obtain the first bounds in \eqref{any} and \eqref{int}, respectively. To simplify the calculation, without loss of generality, by rotating the coordinates, we assume that $x=(\frac{|x|}{\sqrt{d}},\frac{|x|}{\sqrt{d}},\ldots,\frac{|x|}{\sqrt{d}})$.

For any point $\xi=(\xi_1,\xi_2,\ldots,\xi_d)$ and for any $j\in\{1,2,\ldots,d\}$, we consider $\xi_j$ as a complex number with modulus $\eta_j$ and argument (angle) $\psi_j$.
 For a large $R>0$  and $\phi:=\min\{\pi/16, \pi/(16\alpha)\}$, consider the regions in the complex plane:
\begin{align*}
\Gamma_R^{(1)}&=\left\{\eta_0 e^{i\psi}\big|\ \eta_0\in(0,R),\psi\in\left[0,\phi\right]\right\},\\
\Gamma_R^{(2)}&=\left\{\eta_0 e^{i\psi}\big|\ \eta_0\in(0,R),\psi\in\left[\pi-\phi,\pi\right]\right\},
\end{align*}
and denote
$$
C_R^{(1)}=\left\{R e^{i\psi}\big|\ \psi\in \left[0,\phi\right]\right\}\ \text{and}\ C_R^{(2)}=\left\{R e^{i\psi}\big|\ \psi\in \left[\pi-\phi,\pi\right]\right\}.
$$

We calculate the contour integrals of the functions $e^{-t|\xi|^\alpha}e^{i \xi x}|\xi|^{\beta}$ and $e^{-t|\xi|^\alpha}e^{i \xi x}\xi^{\boldsymbol{\beta}}$ on the boundaries of the sectors $\Gamma_R^{(1)}$ and $\Gamma_R^{(2)}$. For the term $|\xi|^a$ in the above two functions, where $a=\alpha$ or $\beta$, we extend it to be a holomorphic function
$$
\left(\sum_{k=1}^d\xi_k^2\right)^{a/2} \quad\text{in}\,\, \mathbb{C}^d,
$$
which needs to be specified by choosing suitable branches.
	On one hand, when $\text{Re}(\xi_j)>0$, we select the branch so that the function $w=z^{a/2}$ maps the sector with angles $[0,2\phi]$ to the sector with angles $[0,a\phi]$.
	On the other hand, when $\text{Re}(\xi_j)<0$, we make the function $w=z^{a/2}$ map the sector with angles $[-2\phi,0]$ to the sector with angles $[-a\phi,0]$.
	
The main idea is to use the contour integrals to equate the integrals on the rays $\psi_j=0,\pi$ and the integrals on the rays $\psi_j=\phi, \pi-\phi$, respectively.
The following are some preliminary calculations on the rays $\psi_j=\frac{\pi}{2}-\sgn( \text{Re}(\xi_j)) \left(\frac{\pi}{2}-\phi\right)$ and the arcs $C_R^{(1)}$ or $C_R^{(2)}$, respectively.  Here $\sgn(\cdot)$  is the sign function.
	
First, we consider the case when $\xi_j$'s are on  the rays $\psi_j=\frac{\pi}{2}-\sgn( \text{Re}(\xi_j)) \left(\frac{\pi}{2}-\phi\right)$, where we can write $\xi_j=\eta_j\exp\left(\frac{\pi i}{2}-\sgn( \text{Re}(\xi_j))\left(\frac{\pi}{2}-\phi\right)i\right)$ with $\eta_j\in [0,R]$. In this case, for any fixed $\xi_k\in \Gamma_R^{(1)}\cup \Gamma_R^{(2)}$, where $k\in \{1,2,\ldots,d\}$, we have
	\be\label{1stf}
\left(\sum_{k=1}^d\xi_k^2\right)^{a/2}=\left(e^{ {2\sgn(\text{Re}(\xi_j))i\pi\phi}} \eta_j^2+\sum_{k\neq j} \xi_k^2 \right)^{a/2},
\ee
where $a=\alpha$ or $\beta$, and
\be\label{2ndf}
e^{i\xi x}=\exp\left( i  \exp\left(\frac{\pi i}{2}-\sgn( \text{Re}(\xi_j))\left(\frac{\pi}{2}-\phi\right)i\right)\eta_j \frac{|x|}{\sqrt d}+\sum\limits_{	k\neq j} i\xi_k \frac{|x|}{\sqrt{d}}\right).
\ee
Notice that if $\psi_k=\frac{\pi}{2}-\sgn( \text{Re}(\xi_k)) \left(\frac{\pi}{2}-\phi\right)$ for all $k\in\{1,2,\ldots,d\}$, it holds that
\be\label{1stf2}
\left(\sum\limits_{k=1}^d\xi_k^2\right)^{{a}/{2}}= \left(\sum\limits_{k=1}^d \eta_k^2 e^{2\sgn(\text{Re}(\xi_k))i\pi\phi}\right)^{{a}/{2}}
\ee
and
\be\label{2ndf2}
e^{i\xi x}=\exp\left( i\sum\limits_{k=1}^d   \exp\left(\frac{\pi i}{2}-\sgn( \text{Re}(\xi_k))\left(\frac{\pi}{2}-\phi\right)i\right)\eta_k \frac{|x|}{\sqrt d}\right).
\ee

Next, we treat the case when $\xi_j$ is on the arc $C_R^{(1)}$ or $C_R^{(2)}$, respectively.

By the definition of the regions $\Gamma_R^{(1)}$ and $\Gamma_R^{(2)}$, for any fixed $\xi_k\in \Gamma_R^{(1)}\cup \Gamma_R^{(2)}$, where $k\neq j$ and $\psi_j\in[0,\phi]\cup [\pi-\phi,\pi]$,  the angle between
$R^2 e^{2i\psi_j}$ and $\sum_{k\neq j}\xi_k^2$
is less than $\pi/2$, so we have
\be\label{es1}\Big|R^2 e^{2i\psi_j}+\sum_{k\neq j}\xi_k^2\Big|\geq |R^2 e^{2i\psi_j}|.
\ee
 Moreover, since $|\arg(\xi_k^2)|\leq 2\phi$ for any $k\neq j$, where $\arg(\cdot)$ is the argument (angle), it follows that
$$
\Big|\arg\Big(R^2 e^{2i \psi_j}+\sum_{k\neq j}\xi_k^2\Big)\Big|\leq 2\phi.
$$
This together with \eqref{es1} implies that
	\be\label{fin}
R^\alpha \cos\left(\alpha \phi\right)  \leq \text{Re} \Big(R^2e^{2i\psi_j}+\sum_{			k\neq j} \xi_k^2\Big)^{{\alpha}/{2}}.
\ee

Now we show that the integral of $ e^{-t\left(\sum_{k=1}^d\xi_k^2\right)^{{\alpha}/{2}}}e^{i \xi x}\left(\sum_{k=1}^d\xi_k^2\right)^{{\beta}/{2}}$ on the arc $C_R^{(1)}$ or $C_R^{(2)}$ tends to $0$ as $R$ tends to infinity.

On the arc $C_R^{(1)}$, we can write $\xi_j=R e^{i \psi_j}$, where $\psi_j\in [0,\phi]$.
By \eqref{1stf}, \eqref{2ndf}, and \eqref{fin}, we have
\be\label{3rdef}
	\begin{aligned}
	&\lim\limits_{R\to \infty}\left|\int_{C_R^{(1)}} e^{-t\left(\sum_{k=1}^d\xi_k^2\right)^{\alpha/2}}e^{i \xi x}\left(\sum_{k=1}^d\xi_k^2\right)^{{\beta}/{2}}d\xi_j\right|\\
	&\leq \lim_{R\to \infty} \int_0^{\phi} \left|\exp\left(-t(R^2 e^{2i\psi_j}+\sum_{
			k\neq j}\xi_k^2)^{\alpha/2}\right)\right| \left|\exp\left( iR e^{i\psi_j}\frac{|x|}{\sqrt{d}}+\sum_{k\neq j} i\xi_k\frac{|x|}{\sqrt{d}}\right)\right|\\
	&\quad \times  \left|\left(R^2 e^{2i\psi_j}+\sum_{k\neq j}\xi_k^2\right)^{\beta/2}\right|
	 \left|iR e^{i\psi_j}\right|d\psi_j\\
	&\leq C  \lim\limits_{R\to \infty} \int_0^{\phi}  e^{-tR^\alpha \cos\left(\alpha\phi\right)} \left(R^{\beta}+\left(\sum_{k\neq j}|\xi_k|^2\right)^{{\beta}/{2}}\right)R\, d\psi_j=0
	\end{aligned}
\ee
for any fixed  $\xi_k\in \Gamma_R^{(1)}\cup \Gamma_R^{(2)}$, where $k\ne j$.

Similarly, on the arc $ C_R^{(2)}$, where $\xi_j=R e^{i \psi_j}$ and $\psi_j\in [\pi-\phi,\pi]$, we have
	\be\label{3rdef2}
	\lim_{R\to \infty}\left|\int_{C_R^{(2)}} e^{-t\left(\sum_{k=1}^d\xi_k^2\right)^{{\alpha}/{2}}}e^{i \xi x}\left(\sum_{k=1}^d\xi_k^2\right)^{{\beta}/{2}}d\xi_j\right|=0
\ee
for any fixed  $\xi_k\in \Gamma_R^{(1)}\cup \Gamma_R^{(2)}$, where $k\neq j$.

Combining \eqref{3rdef} and \eqref{3rdef2} implies that we can apply contour integral to $\xi_j$ if $\xi_k\in \Gamma_R^{(1)}\cup \Gamma_R^{(2)}$ for all $k\neq j$. Therefore, by \eqref{1stf}, \eqref{2ndf}, \eqref{1stf2},  \eqref{2ndf2}, \eqref{3rdef}, and \eqref{3rdef2}, using $d$ times of contour integrals, we infer that
\be\label{int3}
	\begin{aligned}
&\int_{\mathbb{R}^d} e^{-t|\xi|^\alpha}e^{i \xi x}|\xi|^{\beta}\,d\xi\\
&=\sum\limits_{\sgn(\xi_1)=\pm 1}\int_{\mathbb{R}^{d-1}}\int_{0}^\infty \exp\left(-t\left(e^{2i\sgn(\xi_1)\phi} \eta_1^2+\sum\limits_{k=2 }^d \xi_k^2 \right)^{\alpha/2}\right)\\
&\quad \times \exp\left( i  \exp\left(\frac{\pi i}{2}-\sgn( \text{Re}(\xi_1))\left(\frac{\pi}{2}-\phi\right)i\right)\eta_1 \frac{|x|}{\sqrt d}+\sum\limits_{	k=2}^d i\xi_k \frac{|x|}{\sqrt{d}}\right)\\
&\quad \times		\left(e^{2i\sgn(\xi_1)\phi} \eta_1^2+\sum\limits_{k=2 }^d \xi_k^2 \right)^{\beta/2} \exp\left(\frac{\pi i}{2}-\sgn( \text{Re}(\xi_1))\left(\frac{\pi}{2}-\phi\right)i\right)\,d\eta_1 d\xi_2\cdots d\xi_{d}\\
&=\cdots= \sum\limits_{\sgn(\xi_1)=\pm 1}\cdots \sum\limits_{\sgn(\xi_d)=\pm 1}\int_{\mathbb{R}_1^{d}} \exp\left(-t\left(\sum\limits_{k=1}^d e^{2i\sgn(\xi_k)\phi} \eta_k^2\right)^{\alpha/2}\right)\\
&\quad\times \exp\left( i\sum\limits_{k=1}^d   \exp\left(\frac{\pi i}{2}-\sgn( \text{Re}(\xi_k))\left(\frac{\pi}{2}-\phi\right)i\right)\eta_k \frac{|x|}{\sqrt d}\right)\\
&\quad \times		\left(\sum\limits_{k=1}^de^{2i\sgn(\xi_k)\phi} \eta_k^2 \right)^{\beta/2}  \prod_{k=1}^d \exp\left(\frac{\pi i}{2}-\sgn( \text{Re}(\xi_k))\left(\frac{\pi}{2}-\phi\right)i\right)\,d\eta,
\end{aligned}
\ee
where $\mathbb{R}^d_1$ stands for the first quadrant of $\mathbb{R}^d$ and $d\eta=d\eta_1d\eta_2\cdots d\eta_d$. Plugging $$\text{Re}\left(\sum\limits_{k=1}^d e^{2i\sgn(\xi_k)\phi} \eta_k^2\right)^{\alpha/2}\geq |\eta|^\alpha \cos\left(\alpha\phi\right),$$
and
$$
\left|\exp\left( i\sum\limits_{k=1}^d  \exp\left(\frac{\pi i}{2}-\sgn( \text{Re}(\xi_k))\left(\frac{\pi}{2}-\phi\right)i\right)\eta_k \frac{|x|}{\sqrt d}\right)
\right|
=\exp\left(-\sum\limits_{k=1}^d\sin(\phi)\eta_k\frac{|x|}{\sqrt{d}}\right)$$
 into \eqref{int3}, we have
		\be\label{11}
	\begin{aligned}
		&\big|\int_{\mathbb{R}^d} e^{-t|\xi|^\alpha}e^{i \xi x}|\xi|^{\beta}d\xi \big|\leq 2^d \int_{\mathbb{R}_1^{d}} e^{-t|\eta|^\alpha \cos(\alpha\phi)}e^{-\sum\limits_{k=1}^d\sin(\phi)\eta_k|x|/\sqrt{d}}|\eta|^\beta  d\eta\\
		&\leq C\int_{\mathbb{R}^d_1} e^{-t|\eta|^\alpha \cos(\alpha\phi)}e^{- \sum\limits_{k=1}^d \sin(\phi)\eta_k|x|/\sqrt{d}} \sum\limits_{k=1}^d \eta_k^{\beta}d\eta \\
		&\le C\sum\limits_{k=1}^d \int_{0}^{\infty} e^{-t|\eta_k|^\alpha \cos(\alpha\phi) }e^{-  \sin(\phi)\eta_k|x|/\sqrt{d}} \eta_k^{\beta}d\eta_k  \prod\limits_{i\neq k} \int_{\mathbb{R}^{d-1}_1}  e^{-  \sin(\phi)\eta_i|x|/\sqrt{d}} d\eta_i          \\
		&\leq  \frac{C}{|x|^{d-1}} \int_{0}^\infty e^{-t\rho^\alpha \cos(\alpha\phi)}e^{- \sin(\phi)\rho|x|/\sqrt{d}} \rho^{\beta}d\rho=\frac{C}{|x|^{d-1}}\times I,
	\end{aligned}
	\ee
where $$I=\int_{0}^\infty e^{-t\rho^\alpha \cos(\alpha\phi)}e^{-\rho|x|/\sqrt{d}} \rho^{\beta}d\rho\leq \int_{0}^\infty e^{-\rho|x|/\sqrt{d}} \rho^{\beta}d\rho\leq\frac{C^{\beta}}{|x|^{\beta+1}}\Gamma(\beta+1).$$\\
	Therefore, we infer that
\be\label{heatpart1}
\begin{aligned}
	&\big|\int_{\mathbb{R}^d} e^{-t|\xi|^\alpha}e^{i \xi x}|\xi|^{\beta}d\xi \big|\leq  \frac{C_1 C_2^\beta \beta^\beta}{|x|^{\beta+d}}
\end{aligned}
\ee
for some constants $C_1$ and $C_2$, which is the first part on the right-hand side of \eqref{any}.

Finally, we prove \eqref{int}, which is a consequence of the following Claim.
\begin{claim}\label{remark}
For any $\boldsymbol{\beta}=(\beta_1,\ldots,\beta_d)$, where $\beta_i$ is a nonnegative integer,
there exists a constant $C>0$ such that
	 $$
	 \begin{aligned}
	 	&\big|\int_{\mathbb{R}^d} e^{-t|\xi|^\alpha}e^{i \xi x} \xi^{\boldsymbol{\beta}}d\xi \big|\leq C^{|\boldsymbol{\beta}|+\alpha+1}\frac{(\alpha+|\boldsymbol{\beta}|)^{\alpha+|\boldsymbol{\beta}|} t}{|x|^{\alpha+|\boldsymbol{\beta}|+d}}.
	 \end{aligned}
	 $$
\end{claim}
We prove this claim by induction. When $|\boldsymbol{\beta}|=0$, by integration by parts with respect to $\xi_1$,  we  see that
$$
\big|\int_{\mathbb{R}^d} e^{-t|\xi|^\alpha}e^{i \xi x}d\xi\big|=\frac{\alpha \sqrt{d}t}{|x|}
\big|\int_{\mathbb{R}^d} e^{-t|\xi|^\alpha}\frac{\xi_1}{i|\xi|^{2-\alpha}}e^{i \xi x}d\xi\big|.
$$
Then using the method of contour integrals similarly to \eqref{11}, we find that
 $$\big|\int_{\mathbb{R}^d} e^{-t|\xi|^\alpha}\frac{\xi_1}{i|\xi|^{2-\alpha}}e^{i \xi x}d\xi\big|\leq \frac{C}{|x|^{\alpha+d-1}},$$ which implies
$$
\big|\int_{\mathbb{R}^d} e^{-t|\xi|^\alpha}e^{i \xi x}d\xi\big|\leq \frac{Ct}{|x|^{\alpha+d}}.
$$
Without loss of generality, we assume that $\beta_1>0$. For any positive integer $k$, we assume that Claim \ref{remark} is true for any $|\boldsymbol{\beta}|<k$.  When $|\boldsymbol{\beta}|=k$, by integration by parts with respect to $\xi_1$, the induction assumption and \eqref{heatpart1}, it holds that
$$
\begin{aligned}
	& \big|\int_{\mathbb{R}^d} e^{-t|\xi|^\alpha}e^{i \xi x} \xi^{\boldsymbol{\beta}}d\xi \big|\\
	&\leq \big|\frac{\sqrt{d}}{|x|}\int_{\mathbb{R}^d} e^{-t|\xi|^\alpha}e^{i \xi x}\frac{ \beta_1}{i \xi_1} \xi^{\boldsymbol{\beta}} d\xi \big|+\frac{t\alpha\sqrt{d}}{|x|}\big|\int_{\mathbb{R}^d} e^{-t|\xi|^\alpha}e^{i \xi x} \frac{\xi_1}{|\xi|^{2-\alpha}}\xi^{\boldsymbol{\beta}}d\xi \big|\\
	&\leq  \frac{\sqrt{d}}{|x|}C^{\alpha+|\boldsymbol{\beta}|-1}\frac{(\alpha+|\boldsymbol{\beta}|-1)^{\alpha+|\boldsymbol{\beta}|-1}t}{|x|^{\alpha+|\boldsymbol{\beta}|-1+d}}+\frac{t\alpha\sqrt{d}}{|x|}\frac{C_1 C_2^{\alpha+|\boldsymbol{\beta}|-1}(\alpha+|\boldsymbol{\beta}|-1)^{\alpha+|\boldsymbol{\beta}|-1}}{|x|^{\alpha+|\boldsymbol{\beta}|+d-1}}\\
	&\leq C^{\alpha+|\boldsymbol{\beta}|+1}\frac{(\alpha+|\boldsymbol{\beta}|)^{\alpha+|\boldsymbol{\beta}|} t}{|x|^{\alpha+|\boldsymbol{\beta}|+d}}.
\end{aligned}
$$
Thus, we finished the proof of Claim \ref{remark} and therefore completed the proof of Lemma \ref{main}.\\
\end{proof}

Now we are ready to embark on the proof of Theorem \ref{theo02}.
\begin{proof}
By \eqref{de}, the heat kernel $p_\alpha(t,x)$ of the fractional heat equation \eqref{PDE3} satisfies
 $$|\p_t^k p_\alpha(t,x)|=C(d,\alpha) \big|\int_{\mathbb{R}^d} e^{-t|\xi|^\alpha}e^{i \xi x}|\xi|^{\alpha k}d\xi \big|,$$
  which implies \eqref{time}  by part (a) of Lemma \ref{main}. From the first bound $\frac{C_1 C_2^{k\alpha} (k\alpha)^{k\alpha}}{|x|^{k\alpha+d}}$ in \eqref{time}, we see that $p_\alpha$ is of Gevrey class in time of order $\alpha$ when $x\neq 0$. By the second bound $\frac{C}{t^{k+d/\alpha}}\Gamma\left(\frac{k\alpha+d}{\alpha}\right)$ in \eqref{time}, $p_\alpha$ is analytic in time when $t> 0$.

Furthermore,  for any positive integer $k$, by \eqref{de}, we have
$$
|\p_x^{k} p_\alpha(t,x)|\leq C(d,\alpha) \sum\limits_{|\boldsymbol{k}|=k}|\p_x^{\boldsymbol{k}} p_\alpha(t,x)|=C(d,\alpha) \sum\limits_{|\boldsymbol{k}|=k} \big|\int_{\mathbb{R}^d} e^{-t|\xi|^\alpha}e^{i \xi x}\xi^{\boldsymbol{k}} d\xi \big|,
$$ where $\boldsymbol{k}=(k_1,\ldots,k_d)$, $\xi^{\boldsymbol{k}}=\xi_1^{k_1}\ldots \xi_d^{k_d}$, and we sum over all the $\boldsymbol{k}$ satisfying $|\boldsymbol{k}|=k$. By \eqref{x} and the fact that we have $\binom{k+d-1}{d-1}$ choices of $\boldsymbol{k}$ satisfying $|\boldsymbol{k}|=k$, we infer that
$$
|\p_x^{k} p_\alpha(t,x)|\leq C(d,\alpha) \binom{k+d-1}{d-1}\min\left\{\frac{C_1 C_2^{\alpha+k} (\alpha+k)^{\alpha+k} t}{|x|^{\alpha+k+d}},\frac{C}{t^{(k+d)/\alpha}}\Gamma\left(\frac{k+d}{\alpha}\right)\right\},
$$
which implies \eqref{x} for a sufficiently large constant $C_2$. By the bound $\frac{C_1 C_2^{k+\alpha} (k+\alpha)^{k+\alpha} t}{|x|^{\alpha+k+d}}$ in  \eqref{x}, $p_\alpha$ is analytic in space at $|x|\neq 0$. By the other bound
$\frac{C}{t^{(k+d)/\alpha}}\Gamma\left(\frac{k+d}{\alpha}\right)$ in \eqref{x}, $p_\alpha$ is of Gevrey class with order $1/\alpha$ in space when $t>0$ for any $x\in \mathbb{R}^d$.
\end{proof}

\begin{remark}
	Theorem \ref{theo02} is consistent with the fact that the heat kernel of the heat equation $\p_tu-\Delta u=0$ is of Gevrey class of order 2 at $t=0$. Besides, when $\alpha=1$, it is well known that $p_1(t,x)=\frac{Ct}{\left(t^2+|x|^2\right)^{(d+1)/2}}$. By a direct computation, we see that $p_1(t,x)$ satisfies  all the results in Theorem \ref{theo02}.
\end{remark}

We end this section by proving Corollary \ref{add2}.
\begin{proof}
	By Theorem \ref{theo01} and the growth condition \eqref{grow}, we know that there is an unique solution to \eqref{PDE3}:
	$$u(t,x)=\int_{\mathbb{R}^d}p_\alpha(t,x-y)u(0,y)\, dy.$$
	Therefore, by \eqref{x} and \eqref{grow}, we infer that
	$$
	\begin{aligned}
	&|\p_x^k u(t,x)|\leq \int_{\mathbb{R}^d}|\p_x^k p_\alpha(t,x-y)||u(0,y)|\, dy\\
	&\leq \int_{B_1(x)}\frac{C}{t^{(k+d)/\alpha}}\Gamma\left(\frac{k+d}{\alpha}\right) C_1(1+|y|^{\alpha-\epsilon})\,dy\\
	&\quad+\int_{\mathbb{R}^d\backslash B_1(x)}\frac{C_1 C_2^{k+\alpha} (k+\alpha)^{k+\alpha} t}{|x-y|^{\alpha+k+d}}C_1(1+|y|^{\alpha-\epsilon})\,dy\\
	&\leq \frac{C(1+|x|^{\alpha-\epsilon})}{t^{(k+d)/\alpha}}\Gamma\left(\frac{k+d}{\alpha}\right)+\int_{\mathbb{R}^d\backslash B_1(x)}\frac{C^{k+\alpha+1} (k+\alpha)^{k+\alpha} t}{|x-y|^{\alpha+d}}(1+|x|^{\alpha-\epsilon}+|x-y|^{\alpha-\epsilon})\,dy\\
	&\leq \frac{C(1+|x|^{\alpha-\epsilon})}{t^{(k+d)/\alpha}}\Gamma\left(\frac{k+d}{\alpha}\right)+ C^{k+\alpha+2}(k+\alpha)^{k+\alpha}(1+|x|^{\alpha-\epsilon})t,
	\end{aligned}
	$$
	which implies that $u$ is analytic in space when $\alpha\in[1,2)$ and $u$ is of Gevrey class of order $1/\alpha$ in space when $\alpha\in (0,1)$.
\end{proof}

\section{Fractional heat equation on a manifold}\label{4thsec}
In this section, we prove Theorems \ref{theo03} and \ref{theo04} in the setting of $\mathrm{M}$, which is a $d-$dimensional, complete Riemannian manifold.

First we recall a well known lemma.
\begin{lemma}
	 Assume that Condition \eqref{cond2} is satisfied.
	Then for any $D>0$, $\beta\geq 0$, and $t>0$, there exists a positive constant $C$ such that
	\be\label{imp1}
	\int_{\mathrm{M}}\frac{e^{-\frac{Dd(x,y)^2}{t}}}{|B(x,\sqrt{t})|}d(x,y)^{\beta}\,dy\leq Ct^{\beta/2}.
	\ee
\end{lemma}
\begin{proof}
	We give the proof for completeness.
	By Condition \eqref{cond2}, we have
	$$
	\begin{aligned}
	&\int_{\mathrm{M}}\frac{e^{-{Dd(x,y)^2}/{t}}}{|B(x,\sqrt{t})|}d(x,y)^{\beta}\,dy\\
	&=\int_{B(x,\sqrt{t})}\frac{e^{-{Dd(x,y)^2}/{t}}}{|B(x,\sqrt{t})|}d(x,y)^{\beta}\,dy+\int_{\mathrm{M}\backslash B(x,\sqrt{t})}\frac{e^{-{Dd(x,y)^2}/{t}}}{|B(x,\sqrt{t})|}d(x,y)^{\beta}\,dy\\
	&\leq Ct^{\beta/2}+\sum\limits_{k=1}^{\infty}\int_{2^{k-1}\sqrt{t}\leq d(x,y)\leq2^{k}\sqrt{t}}\frac{e^{-{Dd(x,y)^2}/{t}}}{|B(x,\sqrt{t})|}d(x,y)^{\beta}\,dy\\
	&\leq Ct^{\beta/2}+\sum\limits_{k=1}^{\infty}\frac{|B(x,2^k\sqrt{t})|}
	{|B(x,\sqrt{t})|}e^{-D(2^{k-1})^2}(2^k\sqrt{t})^{\beta}\\
	&\leq Ct^{\beta/2}+\sum\limits_{k=1}^{\infty}{C^*}^k e^{-D(2^{k-1})^2}(2^k\sqrt{t})^{\beta}\leq Ct^{\beta/2},
	\end{aligned}
	$$
	where $C^*$ is the constant in Condition \eqref{cond2}.
\end{proof}

 We are ready to prove Theorem \ref{theo03}.
\subsection{Proof of Theorem \ref{theo03}}
\begin{proof}
	It is well known that there is a connection between the heat kernel $E(t,x;y)$ and the fractional heat kernel $p_\alpha(t,x;y)$, which can be found, for instance, in \cite{[B]}, i.e.,
	$$
	\begin{aligned}
	& p_\alpha(t,x;y)=\int_0^\infty E(s,x;y)\eta_t(s)\,ds,
	\end{aligned}
	$$
	where $\eta_t(s)$ is a density function of $\mu_t^\alpha$ satisfying
	$$
	\eta_t(s)=t^{-2/\alpha}\eta_1(t^{{-2}/{\alpha}}s).
	$$
	Therefore,
	\be\label{connection}
	\begin{aligned}
		& p_\alpha(t,x;y)=\int_0^\infty E(s,x;y)t^{-2/\alpha}\eta_1(t^{-2/\alpha}s)\, ds=\int_0^\infty E(t^{2/\alpha}s,x;y)\eta_1(s)\, ds.
	\end{aligned}
	\ee
It is also known that there exists a constant $C$ such that
	\be\label{new2} 0\leq \eta_1(s)\leq Cs^{-1-\alpha/2}e^{-s^{-\alpha/2}},\ee
	which can be found, for instance, in Theorem 3.1 of \cite{[B]}, Theorem 37.1 in \cite{[D]}, or Lemma 1 of \cite{[H]}.

Then for any $t>0$, by \eqref{mild} and \eqref{connection}, it holds that
	\be\label{u0}u(t,x)=\int_{\mathrm{M}}\int_0^\infty E(t^{2/\alpha}s,x;y)\eta_{1}(s)u(0,y)\,dsdy.
	\ee
By Theorem 5.4.12 of \cite{[Saloff]}, Conditions \eqref{cond1} and \eqref{cond2} imply that there exist constants $C$, $d_1$, $d_2$, $D_1$, and $D_2$ such that
\be\label{hea}
\frac{ d_1e^{-D_1 d(x,y)^2/t}}{|B(x,\sqrt{t})|}\leq E(t,x;y)\leq \frac{ d_2e^{-D_2 d(x,y)^2/t}}{|B(x,\sqrt{t})|},
\ee
and
\be\label{Et}
|\p_t E(t,x;y)|\leq \frac{C}{t}\frac{ e^{-D_2 d(x,y)^2/t}}{|B(x,\sqrt{t})|}.
\ee
From \eqref{grow3}, \eqref{u0}, \eqref{hea}, \eqref{imp1}, and \eqref{new2}, we infer that
$$
\begin{aligned}
|u(t,x)|&\leq \int_{\mathrm{M}}\int_0^\infty |E(t^{2/\alpha}s,x;y)|\eta_{1}(s)|u(0,y)|\,dsdy\\
&\leq C\int_{\mathrm{M}}\int_0^\infty \frac{ e^{-{D_2 d(x,y)^2}/(t^{2/\alpha}s)}}{|B(x,\sqrt{t^{2/\alpha}s})|}\eta_{1}(s)(1+d(x,0)^{\alpha-\epsilon}+d(x,y)^{\alpha-\epsilon})\,dsdy\\
&\leq C(1+d(x,0)^{\alpha-\epsilon})\int_0^\infty\eta_1(s)\,ds
+C\int_0^\infty\eta_1(s)(t^{2/\alpha}s)^{(\alpha-\epsilon)/2}\, ds\\
&\leq C(1+d(x,0)^{\alpha-\epsilon})\int_0^\infty\eta_1(s)\,ds+Ct^{\frac{\alpha-\epsilon}{\alpha}}\int_0^\infty s^{-1-\alpha/2}e^{-s^{-\alpha/2}}s^{(\alpha-\epsilon)/2}\, ds\\
&\leq C(1+d(x,0)^{\alpha-\epsilon})+Ct^{(\alpha-\epsilon)/\alpha}.
\end{aligned}
$$	

For any integer $k>0$, we proceed by induction. First, we assume it is true that
	\be\label{as}
	|\p_t^{k-1} u(t,x)|\leq \frac{C^{k} (k-1)^{k-1}}{t^{k-2}}\left(\frac{(1+d(x,0)^{\alpha-\epsilon}}{t}+\frac{1}{t^{\epsilon/\alpha}}\right).
	\ee

	Then for any $t>0$, by \eqref{mild} and \eqref{connection}, it holds that
	\be\label{u1}\p_t^k u(t,x;y)=\int_{\mathrm{M}}\int_0^\infty\p_t E((t-\tau)^{2/\alpha}s,x;y)\eta_{1}(s)\p_\tau^{k-1}u(\tau,y)\,dsdy,\ \forall \tau\in (0,t).
	\ee

	By \eqref{u1}, \eqref{as}, and \eqref{Et}, we have
	\be\label{im1}
	\begin{aligned}
		&|\p_t^k u(t,x;y)|\\
		&\leq \int_{\mathrm{M}}\int_0^\infty\frac{2s}{\alpha}(t-\tau)^{2/\alpha-1}\frac{C}{(t-\tau)^{2/\alpha}s}
\frac{ e^{{-D_2 d(x,y)^2}/((t-\tau)^{2/\alpha}s)}}{|B(x,\sqrt{(t-\tau)^{2/\alpha}s})|}
		\eta_{1}(s)|\p_t^{k-1}u(\tau,y)|\,dsdy\\
		&\leq \frac{C^{k+1/2}(k-1)^{k-1}}{\tau^{k-2}(t-\tau)}\int_{\mathrm{M}}\int_0^\infty\frac{ e^{{-D_2 d(x,y)^2}/((t-\tau)^{2/\alpha}s)}}{|B(x,\sqrt{(t-\tau)^{2/\alpha}s})|} \eta_{1}(s)\left(\frac{1+d(x,0)^{\alpha-\epsilon}}{\tau}+\frac{1}{\tau^{\epsilon/\alpha}}\right)\,dsdy\\
		&\quad+  \frac{C^{k+1/2}(k-1)^{k-1}}{\tau^{k-1}(t-\tau)}\int_{\mathrm{M}}\int_0^\infty\frac{ e^{-D_2 d(x,y)^2/((t-\tau)^{2/\alpha}s)}}{|B(x,\sqrt{(t-\tau)^{2/\alpha}s})|}
		\eta_{1}(s)d(x,y)^{\alpha-\epsilon}\,dsdy\\
		&:=I_1+I_2,
	\end{aligned}
	\ee
where we used the triangle inequality in the second inequality.
	By \eqref{imp1} and \eqref{new2}, we have
	\be\label{im2}
	\begin{aligned}
		I_1&= \frac{C^{k+1/2}(k-1)^{k-1}}{\tau^{k-2}(t-\tau)}
\left(\frac{1+d(x,0)^{\alpha-\epsilon}}{\tau}+\frac{1}{\tau^{\epsilon/\alpha}}\right)\int_0^\infty\int_{\mathrm{M}}\frac{ e^{{-D_2 d(x,y)^2}/((t-\tau)^{2/\alpha}s)}}{|B(x,\sqrt{(t-\tau)^{2/\alpha}s})|}
		\eta_{1}(s)\,dyds\\
		&\leq\frac{C^{k+3/4}(k-1)^{k-1}}{\tau^{k-2}(t-\tau)}\left(\frac{1+d(x,0)^{\alpha-\epsilon}}
	{\tau}+\frac{1}{\tau^{\epsilon/\alpha}}\right)\int_0^\infty \eta_1(s)\,ds\\
	&\leq \frac{C^{k+3/4}(k-1)^{k-1}}{\tau^{k-2}(t-\tau)}\left(\frac{1+d(x,0)^{\alpha-\epsilon}}
	{\tau}+\frac{1}{\tau^{\epsilon/\alpha}}\right),
	\end{aligned}
    \ee
	and
	\be\label{im3}
	\begin{aligned}
		I_2&= \frac{C^{k+1/2}(k-1)^{k-1}}{\tau^{k-1}(t-\tau)}\int_0^\infty\int_{\mathrm{M}}\frac{ e^{{-D_2 d(x,y)^2}/((t-\tau)^{2/\alpha}s)}}{|B(x,\sqrt{(t-\tau)^{2/\alpha}s})|}d(x,y)^{\alpha-\epsilon}\eta_{1}(s)\,dyds\\
		&\leq \frac{C^{k+3/4}(k-1)^{k-1}}{\tau^{k-1}(t-\tau)}\int_0^\infty \left((t-\tau)^{2/\alpha}s\right)^{(\alpha-\epsilon)/2} s^{-1-\alpha/2}e^{-s^{-\alpha/2}}\,ds\\
		&\leq \frac{C^{k+7/8}(k-1)^{k-1}}{\tau^{k-1}(t-\tau)^{\epsilon/\alpha}}.
	\end{aligned}
	\ee
	Now we set $\tau=\frac{(k-1)t}{k}$. Consequently, by plugging \eqref{im2} and \eqref{im3} into \eqref{im1}, we conclude that
	$$
	\begin{aligned}
	&|\p_t^k u(t,x;y)|\\
	&\leq \frac{C^{k+3/4}(k-1)^{k-1}}{\tau^{k-2}(t-\tau)}\left(\frac{1+d(x,0)^{\alpha-\epsilon}}{\tau}+\frac{1}{\tau^{\epsilon/\alpha}}\right)+\frac{C^{k+7/8}(k-1)^{k-1}}{\tau^{k-1}(t-\tau)^{\epsilon/\alpha}}\\
	&\leq \frac{C^{k+1} k^k}{t^{k-1}}\left(\frac{1+d(x,0)^{\alpha-\epsilon}}{t}+\frac{1}{t^{\epsilon/\alpha}}\right),
	\end{aligned}
	$$
	which gives \eqref{aim2} immediately.
\end{proof}

The proof of Theorem \ref{theo04} is divided into two parts: the proof of  \eqref{heatf} and the proof of \eqref{heat3}. We start with the first part in the following subsection.

\subsection{Proof of (\ref{heatf}) in Theorem \ref{theo04}}\label{sub}
\begin{proof}
	By Condition \eqref{cond2}, it is well known that when $r\leq s$,
	\be\label{volume}
	|B(x,r)|\geq\frac{1}{C^*}\left(\frac{r}{s}\right)^{\log_2{C^*}}|B(x,s)|.
	\ee
	See, for example, Remark 4.2.2 of \cite{[Zhangbook]}.
	
	Therefore, by \eqref{connection}, \eqref{hea}, \eqref{new2}, and \eqref{volume}, we have
	\be\label{he1}
	\begin{aligned}
	&p_\alpha(t,x;y)\\
	&\leq \int_{0}^1 \frac{Ce^{-{D_2d(x,y)^2}/(t^{2/\alpha}s)}}{|B(x,\sqrt{t^{2/\alpha}s})|}s^{-1-\alpha/2}e^{-s^{-\alpha/2}}\,ds+\int_{1}^\infty \frac{Ce^{-{D_2d(x,y)^2}/(t^{2/\alpha}s)}}{|B(x,\sqrt{t^{2/\alpha}s})|}s^{-1-\alpha/2}e^{-s^{-\alpha/2}}\,ds\\
	&=  \int_{0}^1 \frac{Ce^{-{D_2d(x,y)^2}/(t^{2/\alpha}s)}}{|B(x,t^{1/\alpha})|}\frac{|B(x,t^{1/\alpha})|}{|B(x,\sqrt{t^{2/\alpha}s})|}s^{-1-\alpha/2}e^{-s^{-\alpha/2}}\,ds\\
	&\quad+\int_{1}^\infty \frac{Ce^{-{D_2d(x,y)^2}/(t^{2/\alpha}s)}}{|B(x,\sqrt{t^{2/\alpha}s})|}s^{-1-\alpha/2}e^{-s^{-\alpha/2}}\,ds\\
	&\leq  \int_{0}^1 \frac{C}{|B(x,t^{1/\alpha})|}\frac{C^*}{s^{{\log_2{C^*}}/{2}}}s^{-1-\alpha/2}e^{-s^{-\alpha/2}}\,ds+\int_{1}^\infty \frac{C}{|B(x,t^{1/\alpha})|}s^{-1-\alpha/2}e^{-s^{-\alpha/2}}\,ds\\
	&\leq \frac{C}{|B(x,t^{1/\alpha})|}.
	\end{aligned}
	\ee
	 If $d(x,y)\geq t^{1/\alpha}$, letting $\xi=\frac{st^{2/\alpha}}{d(x,y)^2}$, again by \eqref{connection}, \eqref{hea}, \eqref{new2}, and \eqref{volume}, we get
	\be\label{he2}
	\begin{aligned}
	p_\alpha(t,x;y)&\leq \int_0^\infty \frac{Ce^{-{D_2}/{\xi}}}{|B(x,\sqrt{\xi}d(x,y))|}\left(\frac{d(x,y)^2\xi}{t^{2/\alpha}}\right)^{-1-\alpha/2}\frac{d(x,y)^2}{t^{2/\alpha}}d\xi\\
	&= \frac{Ct}{d(x,y)^\alpha}\int_{0}^1 \frac{e^{-{D_2}/{\xi}}}{|B(x,\sqrt{\xi}d(x,y))|}\xi^{-1-\alpha/2}d\xi\\
	&\quad+ \frac{Ct}{d(x,y)^\alpha}\int_{1}^\infty \frac{e^{-{D_2}/{\xi}}}{|B(x,\sqrt{\xi}d(x,y))|}\xi^{-1-\alpha/2}d\xi\\
	&\leq \frac{Ct}{d(x,y)^\alpha}\int_{0}^1 \frac{e^{-{D_2}/{\xi}}}{|B(x,d(x,y))|}\frac{|B(x,d(x,y))|}{|B(x,\sqrt{\xi}d(x,y))|}\xi^{-1-\alpha/2}d\xi\\
	&\quad+ \frac{Ct}{d(x,y)^\alpha}\int_{1}^\infty \frac{e^{-{D_2}/{\xi}}}{|B(x,d(x,y))|}\xi^{-1-\alpha/2}d\xi\\
	&\leq \frac{Ct}{d(x,y)^\alpha}\int_0^1 \frac{e^{-{D_2}/{\xi}}}{|B(x,d(x,y))|(\sqrt{\xi})^{\log_2{C^*}}}\xi^{-1-\alpha/2}d\xi+\frac{Ct}{d(x,y)^\alpha|B(x,d(x,y))|}\\
	&\leq \frac{Ct}{d(x,y)^\alpha|B(x,d(x,y))|}.
	\end{aligned}
	\ee
	Thus, we proved the upper bound in \eqref{heatf}.
	
	Now we show the lower bound in \eqref{heatf}.
	By Theorem 3.1 of \cite{[B]}, there exists a constant $s_0=s_0(\alpha)$ such that
	\be\label{eta}\eta_1(s)\geq \frac{\alpha s^{-1-\alpha/2}}{4\Gamma(1-\alpha/2)},\ \forall s>s_0.\ee
	Without loss of generality, we assume that $s_0\geq 1$ in the sequel.
	Then we consider two cases.
	
	When $t^{1/\alpha}\geq d(x,y)$, by \eqref{connection}, \eqref{hea}, \eqref{eta}, and \eqref{volume}, it holds that
	\be\label{geq1}
	\begin{aligned}
		&p_\alpha(t,x;y)=\int_{0}^\infty E(t^{2/\alpha}s,x;y)\eta_1(s)\,ds\\
		&\geq \int_{s_0}^\infty \frac{C d_1e^{-{D_1d(x,y)^2}/(t^{2/\alpha}s)}}{|B(x,\sqrt{t^{2/\alpha}s})|}s^{-1-\alpha/2}\,ds= \int_{s_0}^\infty \frac{C d_1e^{-{D_1d(x,y)^2}/(t^{2/\alpha}s)}}{|B(x,t^{1/\alpha})|}\frac{|B(x,t^{1/\alpha})|}{|B(x,\sqrt{t^{2/\alpha}s})|}s^{-1-\alpha/2}\,ds\\
		&\geq  e^{\frac{-D_1}{s_0}} \int_{s_0}^\infty \frac{C d_1}{|B(x,t^{1/\alpha})|}\frac{1}{C^*s^{{\log_2{C^*}}/{2}}}s^{-1-\alpha/2}\,ds\geq \frac{C}{|B(x,t^{1/\alpha})|}.
	\end{aligned}
	\ee
	
When $t^{1/\alpha}<d(x,y)$, letting $\xi=\frac{st^{2/\alpha}}{d(x,y)^2}$, again by \eqref{connection}, \eqref{hea}, \eqref{eta}, and \eqref{volume}, we have
	\be\label{geq2}
	\begin{aligned}
		p_\alpha(t,x;y)&\geq \int_{s_0}^\infty \frac{C d_1e^{-D_1/\xi}}{|B(x,\sqrt{\xi}d(x,y))|}\left(\frac{d(x,y)^2\xi}{t^{2/\alpha}}\right)^{-1-\alpha/2}\frac{d(x,y)^2}{t^{2/\alpha}}d\xi\\
		&\geq \frac{Ct}{d(x,y)^\alpha}\int_{s_0}^\infty \frac{e^{-D_1/\xi}}{|B(x,d(x,y))|}\frac{|B(x,d(x,y))|}{|B(x,\sqrt{\xi}d(x,y))|}\xi^{-1-\alpha/2}d\xi\\
		&\geq \frac{Ct}{d(x,y)^\alpha}\int_{s_0}^\infty \frac{e^{-{D_1}/{s_0}}}{|B(x,d(x,y))|(\sqrt{\xi})^{\log_2{C^*}}}\xi^{-1-\alpha/2}d\xi\\
		&\geq \frac{Ct}{d(x,y)^\alpha|B(x,d(x,y))|}.
	\end{aligned}
	\ee
	Combining \eqref{geq1} and \eqref{geq2}, we reach  (\ref{heatf}).
\end{proof}

Now in order to prove \eqref{heat3}, we establish an estimate for high-order time derivatives of the heat kernel $E(t,x;y)$ first.
\begin{lemma}\label{heat1}
	Let $\mathrm{M}$ be a $d-$dimensional complete Riemannian manifold satisfying Conditions \eqref{cond1} and \eqref{cond2}. Then for any $x,y\in \mathrm{M}$, $t>0$, and any nonnegative integer $k$, there exist positive constants $C_1$ and $C_2$ such that the heat kernel $E(t,x;y)$ of the heat equation
	$$\p_t u-\Delta u=0$$
	satisfies
	$$
	|\p_t^k E(t,x;y)|\leq \frac{C_1^{k+1}k^{k-2/3}}{t^{k}|B(x,\sqrt{t})|}e^{{-C_2 d(x,y)^2}/{t}}.
	$$
\end{lemma}

\begin{remark}
	To our best knowledge, up to now, in the literature, one can only find the  coarser  bounds
	$$
	|\p_t^k E(t,x;y)|\leq \frac{C(k)}{t^{k}|B(x,\sqrt{t})|}e^{{-C_2 d(x,y)^2}/{t}}
	$$
	in the manifold case, where $C(k)$ is not explicitly calculated.  See, for instance, Theorem 5.4.12 in \cite{[Saloff]}.  Here we obtain a more precise result.
\end{remark}
\begin{proof}
	The proof is similar to Lemma 4.1 of \cite{[Zeng]}. However, since we have different conditions here and we have the estimate of $\p_t^k E(t,x;y)$ for all time $t>0$ instead of $t\in (0,1]$, the proof is a bit different. We present the proof here for the reader's convenience.

	Fix any $t_0>0$ and $x_0, y_0\in \mathrm{M}$. For any nonnegative integer $k$ and $j=1,2,\ldots,k+1$,
	we define
	$$M_j^1=\left\{(t,x):t\in \left(t_0-\frac{j t_0}{2k},t_0\right), d(x,x_0)<\frac{j \sqrt{t_0}}{\sqrt{2k}}\right\},$$
	$$M_j^2=\left\{(t,x):t\in \left(t_0-\frac{(j+0.5) t_0}{2k},t_0\right), d(x,x_0)<\frac{(j+0.5) \sqrt{t_0}}{\sqrt{2k}}\right\}.$$
	Then $M_j^1\subset M_j^2\subset M_{j+1}^1$.
	
	Following  the proof of Lemma 4.1 of \cite{[Zeng]}, for a constant $C$, we have
	\be\label{im4}
	\iint_{M_1^1}|\p_t^k E(t,x;y_0)|^2\, dxdt\leq \frac{C^{2k}k^{2k}}{t_0^{2k}}\iint_{M_{k+1}^1}|E(t,x;y_0)|^2\, dxdt.
	\ee

	Now to estimate the right-hand side of \eqref{im4}, we have two cases.
	
	\noindent\textbf{Case 1:} $\boldsymbol{d(x_0,y_0)\leq \sqrt{4kt_0}}$. In this case, we need to use a well-known result which can be found, for instance, in Lemma 5.2.7 of \cite{[Saloff]}: under Condition \eqref{cond2}, for a constant $C$, we have
	\be\label{vol}
	|B(x,r)|\leq e^{{Cd(x,y)}/{r}}|B(y,r)|,\ \forall x,y\in \mathrm{M}\ \text{and}\ r>0.
	\ee
	By \eqref{hea}, \eqref{volume}, and \eqref{vol}, it holds that
	$$
	\begin{aligned}
	& \frac{C^{2k}k^{2k}}{t_0^{2k}}\iint_{M_{k+1}^1}|E(t,x;y_0)|^2\, dxdt\leq \frac{C^{2k+1/2}k^{2k}|B(x_0,\frac{(k+1)\sqrt{t_0}}{\sqrt{2k}})|}{t_0^{2k-1}\min\limits_{x\in B(x_0,(k+1)\sqrt{t_0}/\sqrt{2k})}|B(x,\sqrt{t_0})|^2}
	\\ &=\frac{C^{2k+1/2}k^{2k}}{t_0^{2k-1}}\frac{|B(x_0,\frac{(k+1)
\sqrt{t_0}}{\sqrt{2k}})|}{|B(x_0,\sqrt{t_0})|^2}\frac{|B(x_0,\sqrt{t_0})|^2}{\min\limits_{x\in B(x_0,(k+1)\sqrt{t_0}/\sqrt{2k})}|B(x,\sqrt{t_0})|^2}\\
	&\leq \frac{C^{2k+3/4}k^{2k}}{t_0^{2k-1}|B(x_0,\sqrt{t_0})|}\left(\frac{k+1}{\sqrt{2k}}\right)^{log_2{C^*}}
\exp\left(\frac{2C(k+1)}{\sqrt{2k}}\right)\leq \frac{C^{2k+1}k^{2k+1}}{t_0^{2k-1}|B(x_0,\sqrt{t_0})|}e^{{-C_2 d(x_0,y_0)^2}/{t_0}}
	\end{aligned}
	$$
	for a constant $C_2$, where we used the condition $d(y_0,x_0)\leq \sqrt{4kt_0}$ in the last inequality.
	
	\noindent\textbf{Case 2:} $\boldsymbol{d(x_0,y_0)>\sqrt{4kt_0}}$. In this case, because $d(x,x_0)<\frac{(k+1)\sqrt{t_0}}{\sqrt{2k}}$ in $M_{k+1}^1$, by the triangle inequality, we have $\frac{\sqrt{2}-1}{\sqrt{2}}< \frac{d(x,y_0)}{d(x_0,y_0)}< 2$. Therefore, by \eqref{hea}, \eqref{volume}, and \eqref{vol}, it holds that
	$$
	\begin{aligned}
	& \frac{C^{2k}k^{2k}}{t_0^{2k}}\iint_{M_{k+1}^1}|E(t,x;y_0)|^2\, dxdt\\
	&\leq \frac{C^{2k}k^{2k}t_0|B(x_0,\frac{(k+1)\sqrt{t_0}}{\sqrt{2k}})|}{t_0^{2k}\min\limits_{x\in B(x_0,(k+1)\sqrt{t_0}/(2\sqrt{k}))}|B(x,\sqrt{t_0})|^2}e^{{-(3-2\sqrt{2})D_2 d(x_0,y_0)^2}/(2t_0)}\\
	&\leq \frac{C^{2k+1/2}k^{2k}}{t_0^{2k-1}}
\frac{|B(x_0,\frac{(k+1)\sqrt{t_0}}{\sqrt{2k}})|}{|B(x_0,\sqrt{t_0})|^2}\frac{|B(x_0,\sqrt{t_0})|^2}{\min\limits_{x\in B(x_0,(k+1)\sqrt{t_0}/(2\sqrt{k}))}|B(x,\sqrt{t_0})|^2}e^{{-C_2 d(x_0,y_0)^2}/{t_0}}\\
	&\leq \frac{C^{2k+3/4}k^{2k}}{t_0^{2k-1}}\frac{1}{|B(x_0,\sqrt{t_0})|}
\left(\frac{k+1}{\sqrt{2k}}\right)^{log_2{C^*}}
\exp\left(\frac{C(k+1)}{\sqrt{k}}\right)e^{{-C_2 d(x_0,y_0)^2}/{t_0}}\\
	&\leq \frac{C^{2k+1}k^{2k+1}}{t_0^{2k-1}|B(x_0,\sqrt{t_0})|}e^{{-C_2 d(x_0,y_0)^2}/{t_0}}
	\end{aligned}
	$$
for a constant $C_2$.
	
	Combining the above two cases, we get
	\be\label{im5}
	\iint_{M_1^1}|\p_t^k E(t,x;y_0)|^2\, dxdt\leq \frac{C^{2k+1}k^{2k+1}}{t_0^{2k-1}|B(x_0,\sqrt{t_0})|}e^{{-C_2 d(x_0,y_0)^2}/{t_0}}.
	\ee
	
	Now we recall a well-known parabolic mean value inequality, which can be found, for instance, in Theorem 14.7 of \cite{[Li]} or Theorem 5.2.9 of \cite{[Saloff]}. For $0<r<R<1$, any nonnegative subsolution $u=u(t,x)$ of the heat equation satisfies
	$$\begin{aligned}
	&\sup _{Q_{r}\left(t_0,x_0 \right)} u(t,x)\leq C\left(\frac{R^2}{|B(x_0,r)|^{2/\nu}}\right)^{\nu/2}
	\left(  \frac{1}{|R-r|^2}\right)^{(\nu+2)/2}\iint_{Q_{R}\left(t_0,x_0\right)} u(t,x)\, d x d t ,
	\end{aligned}$$
	where $\nu>2$ is a constant  and $Q_r(t,x)=(t-r^2,t)\times B(x,r)$. Letting $u(t,x)=|\p_t^k E(t,x;y_0)|^2$, $r\searrow 0$, and $R=\sqrt{t_0/(2k)}$,  using \eqref{volume}, we see that
	\be\label{meanva}
	\begin{aligned}
		|\p_t^k E(t_0,x_0;y_0)|^{2}&\leq \frac{C k}{\left|B\left(x_{0}, \sqrt{t_0/(2k)}\right)\right|t_0} \iint_{Q_{\sqrt{t_0/(2k)}}\left(t_0,x_0\right)} (\p_t^k E(t,x;y_0))^{2}\, d x d t\\
		&=\frac{C k}{|B(x_{0},\sqrt{t_0})|t_0} \frac{|B(x_0,\sqrt{t_0})|}{\left|B\left(x_{0}, \sqrt{t_0/(2k)}\right)\right|}  \iint_{Q_{\sqrt{t_0/(2k)}}\left(t_0,x_0\right)} (\p_t^k E(t,x;y_0))^{2}\, d x d t\\
		& \leq \frac{C k\left(\sqrt{2k}\right)^{log_2(C^*)}}{\left|B(x_{0}, \sqrt{t_0})\right|t_0} \iint_{Q_{\sqrt{t_0/(2k)}}\left(t_0,x_0\right)} (\p_t^k E(t,x;y_0))^{2}\, d x d t.
	\end{aligned}
	\ee
	By \eqref{im5} and \eqref{meanva}, we obtain
	$$
	|\p_t^k E(t_0,x_0;y_0)|^{2}\leq \frac{C^{2k+2}k^{2k+1+\log_2(C^*)/2}}{t_0^{2k}|B(x_0,\sqrt{t_0})|^2}e^{{-C_2 d(x_0,y_0)^2}/{t_0}}.
	$$
	Thus, $$|\p_t^k E(t_0,x_0;y_0)|\leq \frac{C_1^{k+1}k^{k-2/3}}{t_0^{k}|B(x_0,\sqrt{t_0})|}e^{{-C_2 d(x_0,y_0)^2}/{t_0}}$$
for a sufficiently large constant $C_1$, which finishes the proof of Lemma \ref{heat1}.
\end{proof}

To prove the time analyticity of the heat kernel $p_\alpha(t,x;y)$, we use the following result.
\begin{lemma}[\cite{[K]} Proof of Proposition 1.4.2]\label{comp}
	Suppose that $f=f(x)$ is real analytic at $x_0\in \mathbb{R}$, which satisfies near $x_0$,
	$$|f^{(k)}(x)|\leq C_1\frac{k!}{R^k},\quad \forall\  \text{integer}\ k\geq 0.$$
	Assume that $g=g(x)$ is real analytic at $f(x_0)\in \mathbb{R}$ which satisfies near $f(x_0)$,
	$$|g^{(k)}(y)|\leq C_3\frac{k!}{S^k},\quad \forall\  \text{integer}\ k\geq 0.$$
	Here $R$ and $S$ are positive constants.
	Then $h(x)=g(f(x))$ is analytic near $x_0$ and satisfies
	$$|h^{(k)}(x_0)|\leq \frac{C_1C_3}{S+C_1}\frac{k!(1+C_1/S)^k}{R^k},\quad \forall\  \text{integer}\ k\geq 0.$$
\end{lemma}

Now we are ready to prove \eqref{heat3} and thus completes the proof of Theorem \ref{theo04}.
\subsection{Proof of (\ref{heat3}) in Theorem \ref{theo04}}
\begin{proof}
	By \eqref{connection}, we have
	\be\label{connection1}
	\p_t^np_\alpha(t,x;y)=\int_0^\infty \p_t^n E(t^{2/\alpha}s,x;y)\eta_1(s)\,ds.
	\ee
	We write $E(t^{2/\alpha}s,x;y)=E(t,x;y)\circ (t^{2/\alpha}s)=g(t)\circ f(t)$, where
	$g(t):=E(t,x;y)$ and $f(t):=t^{2/\alpha}s$.
	Then by Lemma \ref{heat1}, for a constant $C^{(1)}>0$,
	$$|\p_t^k g(t)|\leq \frac{(C^{(1)})^k k!}{t^k|B(x,\sqrt{t})|}e^{{-C_2 d(x,y)^2}/{t}},\quad \forall\  \text{integer}\ k\geq 0.$$
	Let $C_3=\frac{e^{{-C_2 d(x,y)^2}/(t^{2/\alpha}s)}}{|B(x,\sqrt{t^{2/\alpha}s})|}$ and $S={t^{2/\alpha}s}/{C^{(1)}}$.
	For $f(t)$, it holds that $$|f^{(k)}(t)|\leq \frac{(C^{(2)})^k k! t^{2/\alpha}s}{t^k},\quad \forall\  \text{integer}\ k\geq 0$$
for a constant $C^{(2)}>0$.
	Let $C_1=t^{2/\alpha}s$ and $R={t}/{C^{(2)}}$. Then by Lemma \ref{comp}, we have for a constant $C>0$,
	$$
	|\p_t^kE(t^{2/\alpha}s,x;y)|\leq \frac{C_1C_3}{S+C_1}\frac{k!(1+C_1/S)^k}{R^k}\leq \frac{C^k k!}{t^k}\frac{e^{{-C_2 d(x,y)^2}/(t^{2/\alpha}s)}}{|B(x,\sqrt{t^{2/\alpha}s})|}.
	$$
	Therefore, by \eqref{connection1}, we deduce that
	$$|\p_t^k p_\alpha(t,x;y)|\leq\int_0^\infty \frac{C^k k!}{t^k}\frac{e^{{-C_2 d(x,y)^2}/(t^{2/\alpha}s)}}{|B(x,\sqrt{t^{2/\alpha}s})|} \eta_1(s)\,ds.$$
By the same calculations as \eqref{he1} and \eqref{he2}, we  deduce \eqref{heat3} immediately.
\end{proof}

\section{Corollaries on backward and other equations}\label{5thsec}
In this last section, we present four corollaries, whose statements and proofs are similar to the corresponding results in \cite{[DZ]} and \cite{[Zeng]}.

First we consider the Cauchy problem for the backward nonlocal parabolic equations
\be\label{back}
\left\{\begin{array}{l}
	\partial_{t} u+\mathrm{L}_\alpha^{\kappa} u=0,\ \forall x\in \mathbb{R}^d \\
	u(0, x)=a(x)
\end{array}\right.
\ee
with $\kappa(\cdot,
\cdot)$ satisfying \eqref{1stco} and \eqref{2ndco}.
\begin{corollary}\label{cor1}
	Equation \eqref{back} has a smooth solution $u=u(t,x)$ of polynomial growth of order $\alpha-\epsilon$ in $(0, \delta) \times \mathbb{R}^d$ for some $\delta>0$, i.e.,
	\be\label{aim3}
	|u(t,x)|\leq C(1+|x|^{\alpha-\epsilon}),\ 0<\epsilon<\alpha,\  (t,x)\in (0, \delta)\times \mathbb{R}^d,
	\ee
	if and only if
	\be\label{ifonly3}
	|\left(\mathrm{L}_\alpha^{\kappa}\right)^{k} a(x)| \leq A_1^{k+1}k^k\left(1+|x|^{\alpha-\epsilon}\right), \quad k=0,1,2, \ldots
	\ee
	where $A_1$ is a positive constant.
	
\end{corollary}

\begin{proof}
	
	On  one hand, suppose that (\ref{back}) has a smooth solution of polynomial growth of order $\alpha-\epsilon$, say $u=u(t,x)$. Then $u(-t,x)$ is a solution of the nonlocal parabolic equations with polynomial growth of order $\alpha-\epsilon$. By Theorem \ref{theo01} and \eqref{aim3}, \eqref{ifonly3} follows immediately.
	
	On the other hand, suppose that (\ref{ifonly3}) holds. Then it is easy to check that
	$$u(t,x)=\sum_{j=0}^{\infty} (\mathrm{L}_\alpha^{\kappa})^{j} a(x) \frac{t^{j}}{j !}$$
	is a smooth solution of the fraction heat equation for $t \in (-\delta, 0]$ with $\delta$ sufficiently small. Indeed, the bounds (\ref{ifonly3}) guarantee that the above series and the series
	$$
	\sum_{j=0}^{\infty} (\mathrm{L}_\alpha^{\kappa})^{j+1} a(x) \frac{t^{j}}{j !} \quad \text{and}\quad \sum_{j=0}^{\infty} (\mathrm{L}_\alpha^{\kappa})^{j} a(x) \frac{\p_t t^{j}}{j !}
	$$
	all converge absolutely and uniformly in $[-\delta, 0] \times B_R(0)$ for any fixed $R>0 .$ Hence, $\partial_{t} u-\mathrm{L}_\alpha^{\kappa}u=0$. Moreover, $u$ has polynomial growth of order $\alpha-\epsilon$ since
	\be\label{gro}
	|u(t,x)| \leq \sum_{j=0}^{\infty}\left|(\mathrm{L}_\alpha^{\kappa})^{j} a(x)\right| \frac{t^{j}}{j !} \leq  \sum_{j=0}^{\infty} A_1^{j+1}j^j\left(1+|x|^{\alpha-\epsilon}\right)\frac{t^{j}}{j !}  \leq A_{1} \left(1+|x|^{\alpha-\epsilon}\right)
	\ee
	provided that $t \in[-\delta, 0]$ with $\delta$ sufficiently small.
	Thus, $u(-t,x)$ is a solution to the Cauchy problem of the backward nonlocal parabolic equations (\ref{back}) of polynomial growth of order $\alpha-\epsilon$.
\end{proof}

We have another corollary below about the forward Cauchy problem for the nonlocal parabolic equations
\begin{equation}\label{Cauchy}
\left\{\begin{array}{l}
\partial_{t} u-\mathrm{L}_\alpha^{\kappa} u=0, \ \forall x\in \mathbb{R}^d  \\
u(0, x)=a(x).
\end{array}\right.
\end{equation}
The main point is the analyticity of solutions down to the initial time.
\begin{corollary}
	Equation \eqref{Cauchy}
	 has a smooth solution $u=u(t,x)$ of polynomial growth of order $\alpha-\epsilon$, which is time analytic in
	$[0, \delta) $ for some $\delta>0$ with the radius of convergence independent of $x$ if and only if
	\begin{equation}\label{ifonly4}
	|\left(\mathrm{L}_\alpha^{\kappa}\right)^k a(x)| \leq A_1^{k+1}k^k\left(1+|x|^{\alpha-\epsilon}\right), \quad k=0,1,2, \ldots
	\end{equation}
	for a positive constant $A_1$.
\end{corollary}
\begin{proof}
	On one hand, assuming \eqref{ifonly4}, we can see
	$$u^*(t,x)=\sum_{j=0}^{\infty} (\mathrm{L}_\alpha^{\kappa})^{j} a(x) \frac{t^{j}}{j !}$$
	is a smooth solution to \eqref{Cauchy} for $t \in [0,\delta)$ with $\delta$ sufficiently small.
 Moreover, if $\delta$ is sufficiently small, $u^*$ has polynomial growth of order $\alpha-\epsilon$ by \eqref{gro}, so $u^*$ is the unique solution to \eqref{Cauchy} by part (b) of Theorem \ref{theo01}. 	
	
	By Corollary \ref{cor1}, the backward problem \eqref{back} has a smooth solution $v=v(t,x)$
	in $[0, \delta) \times$ $\mathbb{R}^d$. Define the function $U=U(t,x)$ by
	$$
	U(t,x)=\left\{\begin{array}{ll}
	u^*(t,x), & t \in[0, \delta) \\
	v(-t, x), & t \in(-\delta, 0].
	\end{array}\right.
	$$
	It is straight forward to check that $U(t,x)$ is a solution of the nonlocal parabolic equations in $(-\delta,\delta)\times \mathbb{R}^d$. By Theorem \ref{theo01}, $U(t,x)$ and hence $u(t,x)$ is time analytic at $t=0$ for some $\delta>0$.

	On the other hand, suppose that $u=u(t,x)$ is a solution of the equation (\ref{Cauchy}), which is analytic in
	time at $t = 0$ with the radius of convergence independent of $x$. Then, by definition, $u$ has
	a power series expansion in a time interval $(-\delta, \delta)$, for some $\delta > 0$. Hence (\ref{ifonly4}) holds
	following the proof of Corollary \ref{cor1}.
\end{proof}
\begin{remark}
	Since we have not proved the solution to \eqref{PDE2} is unique, the proofs of the above two corollaries cannot be applied to the manifold case. Therefore, we just restrict the above two corollaries to the case of $\mathbb{R}^d$.
\end{remark}

For the following two corollaries, the operator $\mathrm{L}$ is either $\mathrm{L}_\alpha^{\kappa}$ on $\mathbb{R}^d$, or $\mathrm{L}^\alpha$ on $\mathrm{M}$. For convenience of notation, let $X$ be either $\mathbb{R}^d$ or $\mathrm{M}$ satisfying Conditions \eqref{cond1} and \eqref{cond2}.

Then similar to Theorems 1.4 and 1.5 in \cite{[Zeng]}, we have the following two corollaries.

\begin{corollary}
	Let $p$ be a positive integer and consider the equation
	\be\label{nonlinear2}
	u_t(t,x)-\mathrm{L} u(t,x)=u^p(t,x)\quad \text{in}\,\,(0,1]\times X
	\ee
	with the initial data $u(0,\cdot)$.
	Assume that $u=u(t,x)$ is a mild solution, i.e.,
	$$u(t,x)=\int_{X}p_\alpha(t,x;y)u(0,y)\,dy+\int_{0}^t \int_{X}p_\alpha(t-s,x;y)u^p(s,y)\,dyds$$ and there exists a constant $C_2$ such that
	$$|u(t,x)|\leq C_2,\ \forall (t,x)\in [0,1]\times X.$$
	Then $u$ is time analytic in  $t\in (0,1]$ and the radius of convergence is independent of x.
\end{corollary}

\begin{proof}
	From  \eqref{goal} or \eqref{heat3}, we see by iteration that
	\be\label{ex1}\|\p_t^k p_\alpha(t,x,\cdot)\|_{L^1(X)}\leq C^{k+{1}/{2}}k^{k-2/3}t^{-k},\ \forall\  \text{integer}\ k\geq 0,\ee
	and thus, by the Leibniz rule, it holds that
	\be\label{ex2}\|\p_t^k(t^k p_\alpha(t,x,\cdot))\|_{L^1(X)}\leq C^{k+1}k^{k-2/3},\ \forall\  \text{integer}\ k\geq 0\ee
	for a sufficient large constant $C$.
	
	The rest of the proof is the same as that of Theorem 1.4 in \cite{[Zeng]}.
\end{proof}

\begin{corollary}
	For the equation (\ref{nonlinear2}) with $p$ being any positive rational number, assume that
	$u=u(t,x)$ is a mild solution and there exist constants $C_1$ and $C_2$ such that
	$$0\leq C_1\leq |u(t,x)|\leq C_2, \ \forall (t,x)\in [0,1]\times X.$$
	Then $u$ is time analytic in  $t\in (0,1]$ and the radius of convergence is independent of x.
\end{corollary}
\begin{proof}
 We also have \eqref{ex1} and \eqref{ex2}. Then the rest of the proof is the same as that of Theorem 1.5 in \cite{[Zeng]}.
\end{proof}

\begin{remark}
It is unclear to us whether a similar result holds when $p$ is an irrational number as we are unable to get an appropriate relation between $\p_t^n(t^n u)$ and $\p_t^n(t^n u^p)$, where $n$ is any positive integer. When $p=q_1/q_2$ is a rational number, in Lemma 4.5 of \cite{[Zeng]}, the author used $\p_t^n(t^n u^{1/q_2})$ as a bridge between $\p_t^n(t^n u)$ and $\p_t^n(t^n u^{q_1/q_2})$. Moreover, Lemma \ref{comp} cannot be used directly here. In fact, for any integer $k>0$, if we assume that
$$
|t^n\p_t^n u|\leq N^n n!\quad \forall \text{ positive integer}\ n\le k
$$
for a constant $N>0$, then by Lemma \ref{comp}, we  get
$$|t^k\p_t^k u^p|\leq N^{k+1/2} k!\left(1+\frac{1}{\min|u|}\right)^k,$$
which cannot be used to obtain a positive  radius of convergence.
\end{remark}

\subsection*{Acknowledgement}
H. Dong is partially supported by the NSF under agreement DMS-2055244,
the Simons Foundation, grant \# 709545, and a Simons fellowship. Q.S. Zhang wishes to thank the Simons foundation for its grant 710364.

\end{document}